\documentclass[reqno]{amsart}
\usepackage{amscd, color, mathrsfs,enumerate}
\usepackage[bookmarksnumbered, colorlinks, plainpages]{hyperref}

\usepackage{amsfonts,latexsym,amstext}
\usepackage{amsmath, amssymb,xcolor}
\usepackage{amsthm}
\usepackage{bm}

\newtheorem{theorem}{Theorem}[section]
\newtheorem{lemma}[theorem]{Lemma}
\newtheorem{proposition}[theorem]{Proposition}
\newtheorem{definition}[theorem]{Definition}

\newtheorem{hypothesis}[theorem]{Hypothesis}
\newtheorem{hypotheses}[theorem]{Hypotheses}

\newtheorem{remark}[theorem]{Remark}

\numberwithin{equation}{section}

\def\sqr#1#2{{\vcenter{\vbox{\hrule height .#2pt \hbox{\vrule
 width .#2pt height#1pt \kern#1pt \vrule
width .#2pt} \hrule height .#2pt}}}}

\def\ds{\begin{displaystyle}}
\def\eds{\end{displaystyle}}

\def\<{\langle }
\def\>{\rangle }

\def\R{\mathbb R}
\def\N{\mathbb N}

\allowdisplaybreaks

\title{Space Regularity of Evolution Equations Driven by Rough Paths}

\author[D. Addona]{Davide Addona}
\address{D.A. \& L.L.: Dipartimento di Scienze Matematiche, Fisiche e Informatiche, Plesso di Matema\-ti\-ca, Universit\`a degli Studi di Parma, Viale Parco Area delle Scienze 53/A, I-43124 Parma, Italy}
\author[L. Lorenzi]{Luca Lorenzi}
\author[G. Tessitore]{Gianmario Tessitore}
\address{G.T.: Dipartimento di Matematica e Applicazioni, Universit\`a degli Studi di Milano-Bicocca, Milano, Via R. Cozzi 55, I-20126 Milano Italy}
\email{davide.addona@unipr.it}
\email{luca.lorenzi@unipr.it}
\email{gianmario.tessitore@unimib.it}

\begin{document}

\maketitle

\begin{abstract}
In this paper, we consider the linear evolution equation $dy(t)=Ay(t)dt+Gy(t)dx(t)$, where $A$ is a closed operator, associated to a semigroup, with good smoothing effects in a Banach space $E$, $x$ is a nonsmooth path, which is $\eta$-H\"older continuous for some $\eta\in (1/3,1/2]$, and $G$ is a non-smoothing linear operator on $E$. We prove that the Cauchy problem associated with the previous equation admits a unique mild solution and we also show that the solution increases the regularity of the initial datum as soon as time evolves.
Then, we show that the mild solution is also an integral solution and this allows us to prove a It\^o formula.
\end{abstract}



\section{Introduction} 
Infinite dimensional evolution equations driven by low regularity 'noise type' trajectories were pioneered by Gubinelli and his colleagues in their groundbreaking works \cite{GLT06} and \cite{GT}. These equations are explored in various scenarios, where the singular integral may exhibit enough regularity to be interpreted as a Young integral, or it may require treatment as an integral with respect to a rough path, as defined by T. Lyons \cite{Lyo98}. Despite the huge interest generated by rough integration techniques in the study of stochastic partial differential equations with irregular noises, only recently several authors, including those of this paper, have revisited abstract rough evolution equations in a general Hilbert or Banach framework. They have been particularly interested in the well-posedness and regularity of solutions and in the study of the dynamic systems that these solutions describe. See, for instance, \cite{ALT1}, \cite{ALT2}, \cite{HN}, \cite{HN-1}, \cite{KNA20}, and works closely related to qualitative and asymptotic properties of the dynamical systems, such as \cite{Hof22}, \cite{MaHo23}. 

This paper focuses on a linear evolution equation such as:
\begin{align}
\label{cauchy_prob_rough_intro}
\left\{
\begin{array}{ll}
dy(t)=Ay(t)dt+Gy(t)dx(t),    &  t\in(0,T], \vspace{1mm} \\
y(0)=\psi,  
\end{array}
\right.
\end{align}
where $A:D(A)\subset E\to E$ is a possibly unbounded closed operator on a Banach space $E$ that generates a semigroup of bounded linear operators $(S(t))_{t\geq 0}$  with regularizing properties, $G$ is a non-smoothing linear operator in $E$ and $x$ is an  $\eta$-H\"{o}lder path with $1/3<\eta\leq 1/2$. The initial datum $\psi$ is assumed to be sufficiently regular. Our primary contribution lies in demonstrating that the unique solution increases its regularity in space as it departs from zero. This phenomenon is well-established for classical evolution equations, its validity in the context of rough evolution equations is a natural expectation.
However, as our analysis reveals, the proof of this assertion, even in the linear case, delves into the intricate balance between space and time regularity of the solution.

To elaborate further, we fix a family $(E_\lambda)_{\lambda \geq 0 }$ of Banach subspaces of $E$ with increasing regularity, where $E_0=E$ and $E_1=D(A)$. Our main theorem, stated as Theorem \ref{thm:ex_un_mild_solution_rough}, establishes that if $\psi\in E_{\eta+\alpha}$, with $\alpha>1-2\eta$, equation \eqref{cauchy_prob_rough} admits a unique mild global solution $y$. This solution demonstrates $\eta$-H\"{o}lder continuity in time with values in $E_\alpha$. Our key result reveals that the solution indeed belongs to $E_{1+\mu}$ for every $\mu\in[0,2\eta+\alpha-1)$ as soon as $t>0$ (thus it increases its regularity). 
Furthermore, we provide an estimate for the growth of $|y(t)|_{E_{1+\mu}}$ as $t\searrow 0$. It is worth noting that our result signifies a genuine increase in regularity, as we do not assume that $G$ maps $E_{1+\mu}$ into itself but only into $E_{\eta+\alpha}$, where $\eta+\alpha<1+\mu$.

We study here mild solutions with a convolution integral $\int_0^t S(t-s)G y(s) dx(s)$, constructed utilizing the rough path $(x,\mathbb{X})$, and the SG-derivative of $Gy$ (see Definition \ref{def-2.3}). The notion of SG-derivative employed here is inspired by \cite{GT} and appears to be particularly suitable for proving the regularizing properties of the convolution integral.

Once we have shown that the mild solution is regular and lies in the domain of $A$, the equation can be, when the path $x$ is lifted in a geometric rough path, reformulated in 'classical' or integral sense as 
\begin{eqnarray*}
y(t)=\psi+\int_0^t Ay(s) ds + \int_0^t Gy(s) dx(s),\qquad\;\,t\in [0,T].
\end{eqnarray*}
Eventually, this integral reformulation enables us to establish a chain rule formula, as detailed in Section \ref{sect-5}. We finally recall that, as it is well known, our abstract model covers, for instance,  parabolic PDEs driven by the trajectories of a fractional Brownian motion with Hurst index $H\in (1/3,1/2]$ (see Section \ref{sect-6}).

Similar results have been achieved by the authors of this paper in the context of Young evolution equations (that is when the singular term $x$ is $\eta$-H\"{o}lder with $\eta>1/2$), under two different sets of assumptions regarding the regularity of the initial datum $\psi$.
 
In \cite{HN}, inspired by and extending \cite{GLS}, the authors establish the existence and uniqueness of a local-in-time mild solution to equation \eqref{cauchy_prob_rough}, augmented with an extra drift term. Their framework includes solutions that may blow up at the origin, with the noise being a trace-class fractional Brownian motion with Hurst index $H\in (1/3, 1/2]$ and $G$ being a nonlinear (though regular) term. Unlike our setting, $G$ maps larger spaces $E_{\alpha}$ into smaller ones, demonstrating a smoothing effect. Furthermore, the solution does not exhibit an increase in regularity and lies in $E$. For related results, see also \cite{HN-0}, where certain spatial regularity properties of the mild solution are demonstrated. Specifically, the authors establish that the solution resides in $D((-A)^{\beta})$ with $\beta<\eta$.

In \cite{HN-1}, a semilinear problem for the same equation considered here is analyzed. The main result guarantees global-in-time existence-uniqueness of a mild solution. However, the authors of \cite{HN-1} directly address the Gubinelli derivative, and their results are weaker than ours in terms of the space regularity of the solution. With our notation, their primary outcome ensures the existence of a unique mild solution $y\in C([0,T];E_{\eta+\alpha})$, assuming $\xi\in E_{\eta+\alpha}$. The Gubinelli derivative belongs to $C([0,T];E_{\alpha})\cap C^{\alpha}([0,T];E_{\alpha-\eta})$, where $E_{\alpha-\eta}$ is an extrapolation space. Notably, $\alpha$ can also be negative, implying that the initial condition $\psi\in E_{\eta+\alpha}$  may be quite irregular. Additionally, we refer the reader to \cite{GHN}, where the existence and uniqueness of a local mild solution are established for a more general problem. Here, the operator $A$ is replaced by a family of sectorial operators, dependent on $t$, whose resolvent sets contain a common sector and $G$ may be nonlinear. The authors also demonstrate that the mild solution is a weak solution and continuously depends on the data. Unlike the other cited results, the authors of \cite{GHN} do not impose   that the image of the function $G$ is a proper subspace of its domain.

The paper is structured as follows.
In Section \ref{sect-2}, we introduce the standing assumptions on the operator $A$ and the associated semigroup. Moreover, in  Subsection \ref{subsect-2.1} we recall the construction of the convolution integral $\int_0^tS(t-r)f(r)dx(r)$ when $x$ is $\eta$-H\"older continuous with $\eta>1/2$. A fundamental tool in this direction is the Sewing lemma (see Proposition \ref{prop:new_sew_map}) which we present in the form of \cite{ALT2}. Subsection \ref{subsect-2.1} is a preludio to the definition of the SG-derivative (see Definition \ref{def-2.3}), of a function $y\in C([0,T];E)$, which generalizes the concept of Gubinelli derivative and is modellized on the semigroup $(S(t))_{t\ge 0}$. 
Through the use of the SG-derivative and the sewing Lemma, in Subsection \ref{subsect-2.2} we define the convolution integral $\int_0^tS(t-r)f(r)dx(r)$ also in the case when $x$ is a rough path with $\eta\in (1/3,1/2]$. Some relevant properties of the rough convolution integral and the SG-derivative are also investigated.

The main body of the paper are Sections \ref{sect-3} to \ref{sect-5}. More precisely, in Section \ref{sect-3}, we prove a global in time existence and uniqueness result for the mild solution to problem \eqref{cauchy_prob_rough_intro} and state smoothness properties (in space) of the solution. In Section \ref{sect-4}, we show that 
the solution to problem \eqref{cauchy_prob_rough_intro} can be approximated through solutions to problems associated with ``smooth'' paths $x$. Besides its own interest, this result is used both to prove, in case of geometric rough paths, an integral representation of the mild solution $y$
and, in Section \ref{sect-5}, a It\^{o} formula satisfied by function $y$.

\medskip

\paragraph{\bf Notation} For  every $a,b\in\R$, with $a<b$, we set  
$[a,b]^2_{<}=\{(s,t)\in [a,b]^2: s\le t\}$ and $[a,b]^3_{<}=\{(s,t,u)\in [a,b]^3: s\le t\le u\}$. Given a Banach space $E$ and $a,b\in\R$, with $a<b$, we denote by $B([a,b];E)$ the set of all bounded functions $f:[a,b]\to E$. We endow it with the sup-norm. 
We assume that reader is familiar with its subspace $C^{\alpha}([a,b];E)$ $(\alpha\in (0,1))$. For every $\mu>0$, we denote by $\mathscr C^{\mu}([a,b]^k_<;E)$ the space of functions $f:[a,b]^k_<\to E$ such that $f(t_1,\ldots,t_k)=0$ if $t_i=t_{i+1}$ for some $i=1,\ldots,k-1$ and
\begin{align*}
\|f\|_{\mathscr C^{\mu}([a,b]^k_<;E)}:=\sup_{a<t_1<\ldots<t_k\leq b}\frac{|f(t_1,\ldots,t_k)|}{|t_k-t_1|^\mu}<\infty.    
\end{align*}
If $E=\R$ then we simply write $C^\alpha([a,b])$ and $\mathscr C^{\mu}([a,b]^k_<)$.
The subscript ``$b$'' stands, everywhere it appears, for bounded.
If $G\in\mathscr{L}(E)$, then we denote its norm by $\mathfrak{g}_0$. Similarly, if $G$ is also bounded from $E_{\alpha}$ into itself for some subspace $E_{\alpha}\hookrightarrow E$, then we denote the norm of the restriction of $G$ to $E_{\alpha}$  by $\mathfrak{g}_{\alpha}$. We find it convenient to denote the sum $\mathfrak{g}_{\alpha}+\mathfrak{g}_{\beta}$ by $\mathfrak{g}_{\alpha,\beta}$.
Finally, If $L_1$ and $L_2$ are two bounded linear operators defined on the Banach space $E$, then we denote by $[L_1,L_2]$ their commutator, i.e., $[L_1,L_2]=L_1L_2-L_2L_1$. 

\newpage

\section{Preliminary results and the construction of the rough convolution integral}
\label{sect-2}
\begin{hypotheses}
\label{hyp-main}
\begin{enumerate}[\rm (i)]
\item
$A:D(A)\subset E\to E$ is a closed operator which generates a semigroup of bounded linear operators $(S(t))_{t\geq0}$ on the Banach space $E$.
\item
For every $\lambda\in [0,3)$, there exists a Banach space $E_{\lambda}$ $($with the convention that $E_0=E$ and $E_1=D(A))$ such that $E_{\lambda}$ is continuously embedded into $E_{\zeta}$ if $\zeta<\lambda$. We denote by $K_{\lambda,\zeta}$ a positive constant such that $|x|_{E_\zeta}\le K_{\lambda,\zeta}|x|_{E_\lambda}$ for every $x\in E_{\lambda}$;
\item
for every $\zeta,\lambda\in [0,3)$, with $\zeta\leq\lambda$, $\mu,\nu\in[0,1]$, with $\mu>\nu$, and $T>0$ there exist positive constants $L_{\zeta,\lambda,T}$, and $C_{\mu,\nu,T}$ such that\footnote{When no confusion may arise, we do not stress the dependence of the constants on $T$.}
\begin{align}
\left\{
\begin{array}{ll}
(a)\ \|S(t)\|_{\mathscr{L}(E_\zeta, E_{\lambda})}\leq L_{\zeta,\lambda,T} t^{-\lambda+\zeta},\\[1mm]
(b) \ \|S(t)-I\|_{\mathscr{L}(E_\mu,E_\nu)}\leq C_{\mu,\nu,T} t^{\mu-\nu}
\end{array}
\right.
\label{stime_smgr}
\end{align}
for every $t\in (0,T]$;
\item 
$S(t)$ is injective for every $t\geq 0$.
\end{enumerate}
\end{hypotheses}

\begin{remark}
\label{strong-cont-smgr}
{\rm 
From \eqref{stime_smgr} it follows that, for
every $x\in E$, the function $S(\cdot)x$ is continuous in $(0,\infty)$ with values in $E_{\beta}$ for every $\beta\in [0,3)$. Indeed, fix $t_0>0$, $t\in (t_0,t_0+1)$ and $\alpha\in (0,1)$ such that $\alpha<\beta$. Using the semigroup rule together with \eqref{stime_smgr}, we can estimate
\begin{align*}
|S(t)x-S(t_0)x|_{E_{\beta}}=&|S(t_0)(S(t-t_0)-I)x|_{E_{\beta}}\\
=&|S(t_0/2)(S(t-t_0)-I)S(t_0/2)x|_{E_{\beta}}\\
\le &\|S(t_0/2)\|_{\mathscr{L}(E_0,E_{\beta})}
\|S(t-t_0)-I\|_{\mathscr{L}(E_1,E_0)}|S(t_0)x|_{E_1}\\
\le & 2^{\beta}t_0^{-\beta-1}C_{1,0,t_0+1}L_{0,\beta,t_0+1}L_{0,1,t_0+1}|t-t_0||x|_{E_0}.
\end{align*}
Arguing similarly, we can show that, if $t\in (t_0/2,t_0)$, then
\begin{align*}
|S(t)x-S(t_0)x|_{E_{\beta}}=&|S(t/2)(S(t_0-t)-I)S(t/2)x|_{E_{\beta}}\\
\le &
2^{2\beta+1}t_0^{-\beta-1}C_{1,0,t_0}L_{0,\beta,t_0}L_{0,1,t_0}|t-t_0||x|_{E_0}.
\end{align*}
From these two estimates, the continuity of the function $S(\cdot)x$ at $t_0$ follows at once.

Further, if $x\in E_{\gamma}$ for some $\gamma>0$, then the function $S(\cdot)x$ is continuous in $[0,\infty)$ with values in $E$. Indeed, the continuity at $t=0$ follows immediately from \eqref{stime_smgr}(b), where we take $\mu=\gamma$ and $\nu=0$.}
\end{remark}

Let $y:[a,b]\to E$ and $z:[a,b]^2_<\to E$ be two given functions. We set 
\begin{align*}
(\hat \delta_1y)(s,t)=y(t)-y(s)-\mathfrak a(s,t)y(s)=y(t)-S(t-s)y(s)    
\end{align*}
for every $s,t\in[a,b]^2_<$ and 
\begin{align*}
(\hat\delta_2z)(s,t,u)
= & z(s,u)-z(t,u)-z(s,t)-\mathfrak a(t,u)z(s,t) \\
= & z(s,u)-z(t,u)-S(u-t)z(s,t)    
\end{align*}
for every $(s,t,u)\in[a,b]^3_<$, where $\mathfrak a(s,t)=S(t-s)-I$. Further, we set
\begin{align*}
(\delta_{S,2}y)(s,t,u)=S(t-s)y(s,u)-y(t,u)-S(t-s)y(s,t), \qquad  (s,t,u)\in[a,b]^3_<.
\end{align*}

Throughout the paper, we also need to consider the operators $\delta_1$ and $\delta_2$, which are defined as the corresponding operators $\hat\delta_1$ and $\hat\delta_2$, with $S(t)$ being replaced by the identity operator.

\begin{hypothesis}
\label{hyp:x}
$x\in C^\eta([a,b])$ for some $\eta\in\left(\frac13,\frac12\right ]$.  
\end{hypothesis}

\subsection{Abstract results and Young integral}
\label{subsect-2.1}
Let us explain the idea which leads to the definition of the convolution integral  of the semigroup $(S(t))_{t\ge 0}$ and a continuous function $f:[a,b]\to E$ which satisfies proper assumptions.

If $x$ is smooth (say $x\in C^1([a,b])$), then we may define the convolution integral ${\mathscr I}_{Sf}$ in the classical sense by setting 
\begin{align*}
{\mathscr I}_{Sf}(s,t):=\int_0^t S(t-r)f(r)dx(r)=
\int_0^t S(t-r)f(r)x'(r)dr
\end{align*}
for every $s,t\in[a,b]$ such that $s\le t$. We can also write
\begin{align}
\mathscr{I}_{Sf}(s,t)
= & \int_s^t S(t-r)(f(r)-S(r-s)f(s)+S(r-s)f(s))dx(r) \notag  \\
= & S(t-s)f(s)(x(t)-x(s))+\int_s^t S(t-r)(\hat\delta_1 f)(s,r)dx(r).
\label{esp_int_1}
\end{align}

If $x$ is less regular, then the integral in the right-hand side of \eqref{esp_int_1} is not well-defined, in general. The idea which leads to the definition of the convolution integral when $\eta\in\left (\frac{1}{2},1\right )$ comes from observing that in the 'smooth case' the function $N:[a,b]^2_{<}\to E$, defined by the integral in the last side of \eqref{esp_int_1}, satisfies the condition
$(\hat\delta_2N)(r,s,t)=S(t-s)(\hat\delta_1f)(s,r)(x(t)-x(s))$ for every $(r,s,t)\in [a,b]^3_{<}$, and from the following proposition, which states the existence of a function with the above properties also when $x$ is less smooth, provided that $f$ is smooth enough.

\begin{proposition}
\label{prop:new_sew_map}
Let Hypotheses $\ref{hyp-main}$ be satisfied and fix $\mu>1$, $\beta\in [0,2)$ and $g\in  \mathscr{C}^{\mu}([a,b]^3_<;E_\beta)\cap {\rm Im}(\delta_{S,2})$. Then, there exists a unique function $M_g$, which belongs to $\displaystyle\bigcap_{0\le\varepsilon<1} \mathscr{C}^{\mu-\varepsilon}([a,b]^2_<;E_{\beta+\varepsilon})$ such that $(\hat\delta_2 M_g)(r,s,t)=S(t-s)g(r,s,t)$ for every $(r,s,t)\in [a,b]^3_{<}$. Moreover, for every $\varepsilon\in[0,1)$ there exists a positive constant $C=C(\varepsilon,\beta,\mu,b-a)$, which does not blow up as $b-a$ tends to zero, such that
\begin{align}
\|M_g\|_{\mathscr{C}^{\mu-\varepsilon}([a,b]^2_<;E_{\beta+\varepsilon})}
\leq C\|g\|_{\mathscr{C}^{\mu}([a,b]^3_<;E_\beta)}.
\label{Mg}
\end{align}
Finally, if $h\in\mathscr C^{\mu}([a,b]^2_<;E_{\beta})$, for some $\mu>1$, then $M_{\delta_{S,2}h}(s,t)=S(t-s)h(s,t)$ for every $(s,t)\in[a,b]^2_<$. 
\end{proposition}

\begin{proof}
The first part of the statement has been proved in \cite[Lemma 2.7 \& Proposition 2.8]{ALT2}. Hence, we limit ourselves to proving the last claim.

If we set $\psi(s,t)=S(t-s)h(s,t)$ for every $(s,t)\in[a,b]^2_<$, then from the proof of \cite[Proposition 2.8]{ALT2} it follows that
\begin{align}
M_g(s,t)
= & \psi(s,t)-\lim_{|\Pi_n(s,t)|\to0}\sum_{i=1}^n S(t-t_i)\psi(t_{i-1},t_i) \notag \\
= & \psi(s,t)-\lim_{|\Pi_n(s,t)|\to0}\sum_{i=1}^n S(t-t_{i-1})h(t_{i-1},t_i)    
\label{def_M_espl}
\end{align}
for every $(s,t)\in[a,b]^2_<$,
where $\Pi_n(s,t):=\{s=t_0<t_1<\ldots<t_n=t\}$ is a partition of $[s,t]$ and $|\Pi_n(s,t)|$ denotes its mesh. Since
\begin{align*}
|h(t_{i-1},t_i)|_{E_\beta}
\leq \|h\|_{\mathscr C^{\mu}([a,b]^2_<;E_{\beta})}|t_i-t_{i-1}|^{\mu}, \qquad\;\, i=1,\ldots,n,
\end{align*}
we can estimate (here below $L_{\beta,\beta}=L_{\beta,\beta,b-a}$)
\begin{align*}
\bigg |\sum_{i=1}^n S(t-t_{i-1})h(t_{i-1},t_i)\bigg |_{E_\beta}\leq & L_{\beta,\beta} \|h\|_{\mathscr C^{\mu}([a,b]^2_<;E_{\beta})}\sum_{i=1}^n |t_i-t_{i-1}|^{\mu} \\
\leq & L_{\beta,\beta} \|h\|_{\mathscr C^{\mu}([a,b]^2_<;E_{\beta})}|\Pi_n(s,t)|^{\mu-1}\sum_{i=1}^n (t_i-t_{i-1}) \\
= & L_{\beta,\beta} \|h\|_{\mathscr C^{\mu}([a,b]^2_<;E_{\beta})}(t-s)|\Pi_n(s,t)|^{\mu-1},
\end{align*}
so that the limit in the last side of \eqref{def_M_espl} 
as $|\Pi_n(s,t)|$ tends to zero is equal to zero.
We have so proved that
$M_{\delta_{S,2}h}(s,t)=S(t-s)h(s,t)$ for every $(s,t)\in[a,b]^2_<$, as it has been claimed.
\end{proof}


\begin{remark}
\label{rmk:lin_M}
{\rm 
\begin{enumerate}[\rm (i)]
\item 
From the proof of \cite[Proposition 2.8]{ALT2} it follows that $M_g$ is linear, i.e., if $g_1$ and $g_2$ satisfy the same assumptions as $g$ in Proposition \ref{prop:new_sew_map}, then $M_{g_1+g_2}=M_{g_1}+M_{g_2}$.
\item 
From the proof of the quoted proposition 
it follows also that, if $(S(t))_{t\ge 0}$  satisfies Hypotheses \ref{hyp-main} but (iii)-(a), then, under the same assumption on the function $g$ in Proposition \ref{prop:new_sew_map}, there exists a unique function $M_g\in \mathscr C^{\mu}([a,b]^2_{<};E_{\beta})$ such that $(\hat\delta_2 M_g)(r,s,t)=S(t-s)g(r,s,t)$ for every $(r,s,t)\in [a,b]^3_{<}$. Moreover, estimate \eqref{Mg} holds true with $\varepsilon=0$.
\end{enumerate}}
\end{remark}

Let $g_0:[a,b]^2_{<}\to E$ be the function defined 
by 
\begin{equation}
g_0(s,t)=f(s)(x(t)-x(s)),\qquad\;\,(s,t)\in [a,b]^2_{<}.
\label{g0}
\end{equation}
If $x\in C^1([a,b])$ and $\hat\delta_1 f\in C^{\alpha}([a,b];E)$ for some $\alpha\in (0,1)$, then the function $\delta_{S,2}g_0$ belongs to $\mathscr C^{1+\alpha}([a,b]^3_<;E)$
(indeed, $(\delta_{S,2}g_0)(r,s,t)=-(\hat\delta_1f)(r,s)(x(t)-x(s))$ for every $(r,s,t)\in [a,b]^3_{<}$).
Hence, we can define the convolution integral replacing the term $N$ by the term 
$-M_{\delta_{S,2}g_0}$, so that
\begin{align}
\int_{s}^tS(t-r)f(r)dx(r):=&S(t-s)f(s)(x(t)-x(s))-M_{\delta_{S,2}g_0}(s,t)\notag\\
=&S(t-s)g_0(s,t)-M_{\delta_{S,2}g_0}(s,t)
\label{new-def}
\end{align}
for every $(s,t)\in [a,b]^2_{<}$.

From Proposition \ref{prop:new_sew_map}, we know that $(\hat\delta_2M_{\delta_{S,2}g_0})(r,s,t)=S(t-s)g_0(r,s,t)$ for every $(r,s,t)\in [a,b]^3_{<}$, so that $\hat\delta_2(N+M_{\delta_{S,2}g_0})\equiv 0$ on $[a,b]^3_{<}$. In view of \cite[Proposition 3.4]{GT} this is enough to infer that $M_{\delta_{S,2}g_0}\equiv -N$, implying that the new definition of the convolution integral in \eqref{new-def} coincides with the classical definition.

\subsection{The definition of the rough integral}
\label{subsect-2.2}
If $\eta\in\left (\frac13,\frac12\right ]$, then the expansion in \eqref{esp_int_1} can be used to define the convolution integral only if the continuous function $f:[a,b]\to E$ satisfies the condition $\hat\delta_1f\in \mathscr C^{\alpha}([a,b]^2_{<};E)$ for some $\alpha>1-\eta$. 
Indeed, if $\eta+\alpha\le 1$, we cannot apply Proposition \ref{prop:new_sew_map} to the function
$\delta_{S,2}g_0$. To overcome this difficulty, we need to go further in the expansion of the integral $\mathscr{I}_{Sf}(s,t)$ and introduce the following definition which plays a crucial role in the construction of the convolution integral.

\begin{definition}
\label{def-2.3}
Given a continuous function $f:[a,b]\to E$, we say that $f$ admits SG-derivative with respect to $x$, with remainder of order $\rho>\eta$, if there exist $\check{f}\in C([a,b];E)$ and $R\in \mathscr C^\rho([a,b]^2_<;E)$
such that
\begin{align}
\label{formula_S_der}
(\hat\delta_1 f)(s,t)=S(t-s)\check f(s)(x(t)-x(s))+R(s,t)    
\end{align}
for every $(s,t)\in[a,b]^2_<$ and
\begin{equation}
\limsup_{t\to s^+}\frac{|x(t)-x(s)|}{|t-s|^{\rho}}=\infty,\qquad\;\,s\in [a,b).
\label{cond-x}
\end{equation}
The function $\check f$ is called SG-derivative of $f$ with respect to $x$ with remainder $R$.

\end{definition}

\begin{lemma}
Let $f:[a,b]\to E$ be a continuous function. Then, $f$ admits at most one SG-derivative. In particular, if $f\in C([a,b];E)$ is such that $\hat{\delta}_1f\in \mathscr C^{\rho}([a,b]^2_{<};E)$ for some $\rho>\eta$ and condition \eqref{cond-x} is satisfied, then $f$ admits SG-derivative, which identically vanishes in $[a,b]$.
\end{lemma}

\begin{proof}
Assume that $f$ admits SG-derivatives $\check{f}_1$ and $\check{f}_2$, and let $\rho_1,\rho_2>\eta$ be such that
\begin{eqnarray*}
&(\hat\delta_1f)(s,t)=S(t-s)\check{f}_1(s)(x(t)-x(s))+R_1(s,t),\\
&(\hat\delta_1f)(s,t)=S(t-s)\check{f}_2(s)(x(t)-x(s))+R_2(s,t)
\end{eqnarray*}
for every $(s,t)\in [a,b]^2_{<}$ and some $R_1\in\mathscr{C}^{\rho_1}([a,b]^2_{<})$
and $R_2\in\mathscr{C}^{\rho_2}([a,b]^2_{<})$.

Note that
\begin{eqnarray*}
\limsup_{t\to s^+}\frac{|x(t)-x(s)|}{|t-s|^{\rho}}=\infty
\end{eqnarray*}
for every $s\in [a,b)$, where $\rho=\min\{\rho_1,\rho_2\}$. Let us fix $s\in [a,b)$ and let $(t_n)_{n\in\N}\in (s,b)$ be a sequence converging to
$s$ as $n$ tends to $\infty$ and such that 
\begin{equation}
\lim_{n\to\infty}\frac{|x(t_n)-x(s)|}{|t_n-s|^{\rho}}=\infty.
\label{expl_x}
\end{equation}

Set $y=\check f_1-\check f_2$. From the definition of SG-derivative, it follows that
\begin{align*}
|S(t_n-s)y(s) 
(x(t_n)-x(s))|_E=O(|t_n-s|^\rho) 
\end{align*}
for every $n\in\N$, i.e., 
\begin{align*}
\sup_{n\in\N}\bigg (|S(t_n-s)y(s)
|_E
\frac{|x(t_n)-x(s)|}{|t_n-s|^{\rho}}\bigg )<\infty
\end{align*}
and \eqref{expl_x} implies that 
$|S(t_n-s)y(s)|_E$ vanishes as $n$ tends to $\infty$.
Let us show that this gives $y(s)=0$. From 
properties \eqref{stime_smgr} it follows that
\begin{align*}
|S(\sigma)y(s)|_E
\le &|S(t_n-s)S(\sigma)y(s)|_E
+|S(t_n-s)S(\sigma)y(s)-S(\sigma)y(s)|_E\\
\le &\|S(\sigma)\|_{\mathscr{L}(E)}|S(t_n-s)y(s)|_E
+|t_n-s|^\beta|S(\sigma)y(s)|_{E_\beta}\\
\le &
\|S(\sigma)\|_{\mathscr{L}(E)}|S(t_n-s)y(s)|_E
+L_{0,\beta,b}\sigma^{-\beta}\|y\|_{C([a,b];E)}|t_n-s|^{\beta}
\end{align*}
for every $n\in\N$ and $\sigma>0$. Letting $n$ tend to $\infty$, we conclude that $S(\sigma)y(s)=0$ for every $\sigma>0$. Since each operator $S(\sigma)$ is one to one, we deduce that $y(s)=0$. Hence, 
$y$ identically vanishes in $[a,b)$ and, by continuity, it vanishes also at $t=b$. We have so proved that the SG-derivative of $f$ is unique, if existing.

To complete the proof, we observe that, if $\hat{\delta}_1f\in \mathscr C^{\rho}([a,b]^2_{<};E)$, then formula \eqref{formula_S_der} is clearly satisfied by taking $\check{f}\equiv 0$.
\end{proof} 

\begin{remark}
\label{rmk:lin_T_der}
{\rm 
\begin{enumerate}[\rm (i)]
\item
From the linearity of \eqref{formula_S_der}, it follows that, if $f_1$ and $f_2$ admit SG-derivative $\check f_1$ and $\check f_2$ with remainder $R^{F_1}$ and $R^{F_2}$, respectively, then $f=f_1+f_2$ admits SG-derivative $\check f=\check f_1+\check f_2$ and remainder $R^F=R^{F_1}+R^{F_2}$.
\item
In the particular case when $S(t)=I$ for every $t\ge 0$, Definition \ref{def-2.3} coincides with
the definition of Gubinelli derivative introduced in \cite{G04} (apart from the extra condition \eqref{cond-x} that we are assuming here).
Indeed, a function $f:[a,b]\to E$ admits Gubinelli derivative if there exist $\rho>\eta$ and a function $\check f_G\in C([a,b];E)$ and $\widetilde R^F\in\mathscr C^{\rho}([a,b]^2_<;E)$ such that
\begin{align*}
f(t)-f(s)
= \check f_G(s)(x(t)-x(s))+\widetilde R^F(s,t), \qquad\;\, (s,t)\in[a,b]^2_<.
\end{align*}
Clearly, $f(t)-f(s)=(\hat\delta_1f)(s,t)$ when $S(t)=I$ for every $t\in[a,b]$, so that $\check{f}$ coincides with $\check{f}_G$.
\end{enumerate}
}
\end{remark}

Let us go back to the definition of the rough convolution integral and assume that $f$ admits SG-derivative $\check f$ with remainder $R^F$. Again, we assume that $x\in C^1([a,b])$ in order to rigorously justify all the terms which will appear in the forthcoming computations.
In such a case, we can replace the term $(\hat\delta_1f)(s,r)$ in the integral, which defines the function $N$, in the last side of \eqref{esp_int_1}, with the sum
$S(t-r)\check{f}(s)(x(r)-x(s))+R^F(s,r)$, obtaining
\begin{align}
({\mathscr I}_{Sf})(s,t)    
= & S(t-s)f(s)(x(t)-x(s))+S(t-s)\check{f}(s)\int_s^t(x(r)-x(s))dx(r) \notag \\
& + \int_s^t S(t-r)R^F(s,r)dx(r)
\label{conti_int_rough_reg}
\end{align}
for every $(s,t)\in[a,b]^2_<$.
If we set
\begin{align}
&\mathbb X(s,t):=\int_s^t(x(r)-x(s))dx(r)=\frac{1}{2}(x(t)-x(s))^2,\label{funct-X}\\
&g_{\check{f}}(s,t):=(x(t)-x(s))f(s)+\mathbb X(s,t)\check{f}(s)
\label{funct-A}
\end{align}
for every $(s,t)\in[a,b]^2_<$, then we can rewrite \eqref{conti_int_rough_reg} in the following more compact form:
\begin{align*}
({\mathscr I}_{Sf})(s,t)
= S(t-s)g_{\check{f}}(s,t)+\int_s^t S(t-r)R^F(s,r)dx(r), \qquad\;\, (s,t)\in[a,b]^2_<.
\end{align*}

Following the idea of the case $x\in C^{\eta}([a,b])$, with
$\eta\in\left (\frac{1}{2},1\right )$, we formally replace the integral term in the right-hand side of the previous formula, which now has no meaning in general, with the term $-M_{\delta_{S,2}g_{\check{f}}}(s,t)$, and still formally we define
\begin{equation}
\mathscr{I}_{Sf}(s,t)=S(t-s)g_{\check{f}}(s,t)-M_{\delta_{S,2}g_{\check{f}}}(s,t),\qquad\;\, (s,t)\in [a,b]^2_{<}.
\label{form-formale}
\end{equation}

To make the previous formal definition rigorous, we begin by
observing that the function defined in the last side of \eqref{funct-X} satisfies the following equality
\begin{align}
\mathbb{X}(s,u)-\mathbb{X}(s,t)-\mathbb{X}(t,u)=(x(t)-x(s))(x(u)-x(t)),\qquad a\le s\le t\le u\le b.
\label{form-diff-X}
\end{align}



\begin{remark}
\label{rem-country}
{\rm Note that all the functions
which satisfy the equality \eqref{form-diff-X} can be obtained adding to the function in the last side of \eqref{funct-X}
the term $\delta_1f$ for some $f\in C^{2\eta}([0,T])$. Indeed, 
condition \eqref{form-diff-X} can be written in the more compact form as $(\delta_2\mathbb{X})(s,t,u)=(x(t)-x(s))(x(u)-x(t))$ for every $(s,t,u)\in [0,T]^3_<$. Therefore, $\mathbb{X}$ and $\mathbb{X}'$ satisfy \eqref{form-diff-X} if and only if their difference belongs to the kernel of the operator $\delta_2$, which coincides with the image of the operator $\delta_1$. In fact, it is immediate to check that the image of $\delta_1$ is contained in the kernel of $\delta_2$. Suppose that $Z$ belongs to the kernel of $\delta_2$. Then, in particular $(\delta_2Z)(a,s,t)=0$ for every $(s,t)\in [a,b]^2_{<}$ or, equivalently,
\begin{equation}
Z(a,t)=Z(s,t)+Z(a,s),\qquad\;\,
(s,t)\in [a,b]^3_{<}.
\label{easy}
\end{equation} 

Define the function $f:[a,b]\to E$ by setting $f(t)=Z(a,t)$ for every $t\in [a,b]$. Then,
\begin{eqnarray*}
(\delta_1f)(s,t)=Z(a,t)-Z(a,s)=Z(s,t),\qquad\;\,(s,t)\in [a,b]^2_{<}, 
\end{eqnarray*}
where the last inequality follows from
\eqref{easy}. Hence, $Z$ belongs to the image of $\delta_1$.

If $\eta\le\frac{1}{2}$, then we conclude that there are infinitely many functions which satisfy condition \eqref{form-diff-X}. On the contrary, if $\eta>\frac12$ then $\mathbb X$ is uniquely determined. Indeed, in such a case the difference of two functions which satisfy Hypothesis \ref{hyp:X_2} is equal to $\delta_1f$ for some function $f\in C^{2\eta}([a,b])$. Since $2\eta>1$, $f$ is constant and $\delta_1f\equiv 0$.
}\end{remark}

For the construction of the convolution integral, we need to fix a function $\mathbb{X}\in\mathscr{C}^{2\eta}([0,T]^2_{<})$, which satisfies condition \eqref{form-diff-X}.

\begin{hypothesis}
\label{hyp:X_2}
$\mathbb X\in \mathscr C^{2\eta}([a,b]^2_<)$ is any, but fixed, increment which satisfies condition \eqref{form-diff-X}.
\end{hypothesis}

As a second step to make formula \eqref{form-formale} rigorous, we observe that, for every $(s,t,u)\in[a,b]^3_<$,
\begin{align}
(\delta_{S,2}g_{\check{f}})(s,t,u)
= & S(t-s)g_{\check{f}}(s,u)-g_{\check{f}}(t,u)-S(t-s)g_{\check{f}}(s,t) \notag \\
=& (S(t-s)f(s)-f(t))(x(u)-x(t)) - \check{f}(t)\mathbb X(t,u)\notag\\
&+S(t-s)\check{f}(s)(\mathbb{X}(s,u)-\mathbb{X}(s,t))\notag\\
=& (S(t-s)f(s)-f(t))(x(u)-x(t)) - \check{f}(t)\mathbb X(t,u) \notag \\
& +S(t-s)\check{f}(s)(\mathbb X(t,u)+(x(t)-x(s))(x(u)-x(t)))\notag \\
= & [-(\hat\delta_1f)(s,t)+S(t-s)\check{f}(s)(x(t)-x(s))](x(u)-x(t)) \notag \\
& -(\hat\delta_1\check f)(s,t)\mathbb X(t,u) \notag \\
= & -R^F(s,t)(x(u)-x(t))-(\hat\delta_1\check{f})(s,t)\mathbb X(t,u),
\label{S_incr_A}
\end{align}
where we have used formula \eqref{form-diff-X} and, in the last equality, we have applied the definition of SG-derivative. The computations in \eqref{S_incr_A} yield the following result.

\begin{lemma}
\label{lemma:def_rough_conv_int}
Under Hypotheses $\ref{hyp:x}$ and $\ref{hyp:X_2}$, assume that
$f\in C([a,b];E)$ admits SG-derivative $\check{f}$ with remainder $R^F\in \mathscr C^\rho([a,b]^2_<;E_\lambda)$ and that $\hat\delta_1\check{f}\in \mathscr C^\zeta([a,b]^2_<;E_{\omega})$ for some nonnegative parameters $\rho,\lambda,\omega,\zeta$. Then, the function $\delta_{S,2}g_{\check{f}}$ belongs to $\mathscr{C}^{\mu}([a,b]^3_{<};E_{\beta})$ where $\mu=(\rho+\eta)\wedge(\zeta+2\eta)$ and $\beta=\lambda\wedge\omega$. Moreover, there exists a positive constant $c_*$, which depends on
$\lambda$, $\omega$, $\rho$, $\eta$, $\zeta$ and $b-a$ $($and it does not blow up as $b-a$ tends to zero$)$, such that
\begin{align}
&\|\delta_{S,2}g_{\check{f}}\|_{\mathscr{C}^{\mu}([a,b]^3_{<};E_\beta)}\notag\\
\leq & c_*\big([x]_{C^\eta([a,b])}\|R^F\|_{\mathscr C^\rho([a,b]^2_<;E_\lambda)}+\|\mathbb X\|_{\mathscr C^{2\eta}([a,b]^2_<)}\|\hat\delta_1\check{f}\|_{\mathscr C^\zeta([a,b]^2_<;E_\omega)}\big).
\label{stima-A}
\end{align} 
\end{lemma}

In view of Lemma \ref{lemma:def_rough_conv_int} and Proposition \ref{prop:new_sew_map}, the function $M_{\delta_{S,2}g_{\check{f}}}$, and consequently formula \eqref{form-formale}, is now completely meaningful if $(\rho+\eta)\wedge(\zeta+2\eta)>1$. Thus, we can give the following definition.

\begin{definition}
\label{def:def_rough_conv_int}
Let $f\in C([a,b];E)$ admit SG-derivative $\check{f}$ with remainder $R^F\in \mathscr C^\rho([a,b]^2_<;E_\lambda)$ and $\hat\delta_1\check{f}\in \mathscr C^\zeta([a,b]^2_<;E_{\omega})$, for some
$\lambda,\omega\ge 0$ and $\rho,\zeta>0$, such that
$(\rho+\eta)\wedge(\zeta+2\eta)>1$. We define the $($rough$)$ convolution integral $\mathscr I_{Sf}:[a,b]^2_<\to E$ with respect to the path $X=(x,\mathbb{X})$ as
\begin{align}
\label{def_int_conv_rough}
{\mathscr I}_{Sf}(s,t)=S(t-s)g_{\check{f}}-M_{\delta_{S,2}g_{\check{f}}}(s,t),  \qquad\;\, (s,t)\in[a,b]^2_<,  
\end{align}
where $M_{\delta_{S,2}g_{\check f}}$ is the $($unique$)$ function defined in Proposition $\ref{prop:new_sew_map}$. 
\end{definition}

\begin{remark}
{\rm Different choices of the function $\mathbb{X}$ lead to different values of the convolution integral. For example, if we choose $\mathbb{X}(s,t)=\frac{1}{2}(x(t)-x(s))^2$ for every $(s,t)\in [a,b]^2_{<}$ and $x$ is a trajectory of a one dimensional Brownian motion, then \eqref{def_int_conv_rough} defines a trajectory of the Stratonovich convolution integral. On the other hand, if we take $\mathbb{X}(s,t)=\frac{1}{2}(x(t)-x(s))^2+\frac{t-s}{2}$ for every $(s,t)\in [a,b]^2_{<}$ and $x$ is as above, then \eqref{def_int_conv_rough} defines a trajectory of It\^{o} convolution integral.}    
\end{remark}

As a consequence of formula \eqref{S_incr_A} and Lemma \ref{lemma:def_rough_conv_int}, we obtain the following estimate.

\begin{lemma}
\label{lem:reg_int_conv_rough}
Under Hypotheses $\ref{hyp:x}$ and $\ref{hyp:X_2}$, fix $\alpha\in (1-2\eta,1)$, $\lambda\in [0,2\eta+\alpha-1)$, $\omega\in [\lambda,1)$ and $\omega_1\in [\alpha,\eta+\alpha]$. Further assume that
$f\in C([a,b];E)\cap B([a,b];E_{\eta+\alpha})$ admits SG-derivative $\check{f}\in C([a,b];E)\cap B([a,b];E_{\omega_1})$ with remainder $R^F\in \mathscr C^{\eta+\alpha-\lambda}([a,b]^2_<;E_\lambda)$ and that $\hat\delta_1\check{f}\in \mathscr C^{\alpha}([a,b]^2_<;E_{\omega})$.
Then, 
$\mathscr I_{Sf}$ belongs to $\mathscr C^{\theta}([a,b]^2_<;E_\xi)$ for every $\xi\in[\lambda,\lambda+1)$, where $\theta=\eta-(\xi-\eta-\alpha)^+$. Moreover, for every $\xi\in[\lambda,\lambda+1)$ there exists a positive 
constant $C$, which depends on $\eta$, $\alpha$, $\xi$, $\lambda$, $\omega$, $\omega_1$ and $b-a$, and does not blow up as $b-a$ tends to zero, such that
\begin{align}
\|\mathscr I_{Sf}\|_{\mathscr C^{\theta}([a,b]^2_<;E_\xi)}
\leq   C\Big(& \|f\|_{B([a,b];E_{\eta+\alpha})}
[x]_{C^\eta([a,b])}
\notag \\
& +\|\check f\|_{B([a,b];E_{\omega_1})}
\|\mathbb X\|_{\mathscr C^{2\eta}([a,b]_<^2)}(b-a)^{\eta-(\xi-\omega_1)^++(\xi-\eta-\alpha)^+}  \notag \\
& +[x]_{C^\eta([a,b])}\|R^F\|_{\mathscr C^{\eta+\alpha-\lambda}([a,b]^2_<;E_\lambda)}(b-a)^{2\eta+\alpha-(\xi\vee\lambda)-\theta}\notag\\
&+\|\mathbb X\|_{\mathscr C^{2\eta}([a,b]^2_<)}\|\hat\delta_1\check{f}\|_{\mathscr C^\alpha([a,b]^2_<;E_\omega)}(b-a)^{2\eta+\alpha-(\xi\vee\lambda)-\theta}\Big ).
\label{stima_int_completa}
\end{align}
\end{lemma}

\begin{proof}
To prove \eqref{stima_int_completa}, we estimate the two terms which appear in the definition of ${\mathscr I}_{Sf}$ (see \eqref{def_int_conv_rough}). For this purpose, we fix $\xi\in[0,\lambda+1)$ and, taking  \eqref{stime_smgr} into account, we get
\begin{align*}
|S(t-s)g_{\check f}(s,t)|_{E_{\xi}}\le & 
|S(t-s)f(s)(x(t)-x(s))|_{E_\xi}
+|S(t-s)\check f(s)\mathbb X(s,t)|_{E_\xi}\notag\\
\le & c_{\alpha,\eta,\xi}\|f\|_{B([a,b];E_{\eta+\alpha})}[x]_{C^\eta([a,b])}|t-s|^{\eta-(\xi-\eta-\alpha)^+}\notag\\
&+c_{\omega_1,\xi}\|\check f\|_{B([a,b];E_{\omega_1})}\|\mathbb X\|_{\mathscr C^{2\eta}([a,b]_<^2)}|t-s|^{2\eta-(\xi-\omega_1)^+}
\end{align*}
for every $(s,t)\in[a,b]^2_<$ and some positive constants $c_{\alpha,\eta,\xi}$ and $c_{\omega_1,\xi}$, so that
the function $(s,t)\mapsto S(t-s)g_{\check f}(s,t)$ belongs to
$\mathscr{C}^{\theta}([a,b]^2_{<};E_{\xi})$ for every $\xi\in [0,\lambda+1)$ and
\begin{align}
\frac{|S(t-s)g_{\check f}(s,t)|_{E_{\xi}}}{|t-s|^{\theta}}
\le & c_1\|f\|_{B([a,b];E_{\eta+\alpha})}[x]_{C^\eta([a,b])}\notag\\
&+c_1\|\check f\|_{B([a,b];E_{\omega_1})}\|\mathbb X\|_{\mathscr C^{2\eta}([a,b]_<^2)}(b-a)^{\eta-(\xi-\omega_1)^++(\xi-\eta-\alpha)^+}
\label{stima-If-1}
\end{align}
for every $(s,t)\in [a,b]^2_{<}$ and some positive constant $c_1$, which depends on $\alpha$, $\eta$, $\xi$, $\omega_1$ and $b-a$.

Finally, from \eqref{stima-A} and \eqref{Mg} we deduce that
\begin{align}
&\|M_{\delta_{S,2}g_{\check f}}\|_{\mathscr C^{2\eta+\alpha-\lambda-\varepsilon}([a,b]^2_<;E_{\lambda+\varepsilon})}
\notag\\
\leq &  c_2\big([x]_{C^\eta([a,b])}\|R^F\|_{\mathscr C^{\eta+\alpha-\lambda}([a,b]^2_<;E_\lambda)}+\|\mathbb X\|_{\mathscr C^{2\eta}([a,b]^2_<)}\|\hat\delta_1\check{f}\|_{\mathscr C^\alpha([a,b]^2_<;E_\omega)}\big)
\label{stima_M}
\end{align}    
for every $\varepsilon\in [0,1)$ and some positive constant $c_2$, which depends on $\varepsilon$,
$\lambda$, $\omega$, $\alpha$, $\eta$  and $b-a$, and it does not blow up as $b-a$ tends to zero. 

Now, we distinguish the cases $\xi\in [0,\lambda]$ and $\xi\in (\lambda,\lambda,1)$. In the former case,
recalling that $E_{\xi}$ is continuously embedded into $E_{\lambda}$, from
\eqref{stima_M} it follows that
$M_{\delta_{S,2}g_{\check f}}\in\mathscr{C}^{2\eta+\alpha-\lambda}([a,b]^2_{<};E_{\xi})$. In the latter case, taking $\varepsilon=\xi-\lambda$ in \eqref{stima_M}, we conclude that $M_{\delta_{S,2}g_{\check f}}\in\mathscr{C}^{2\eta+\alpha-\xi}([a,b]^2_{<};E_{\xi})$. In both the cases, we can estimate
\begin{align}
&\|M_{\delta_{S,2}g_{\check f}}\|_{\mathscr{C}^{2\eta+\alpha-(\xi\vee\lambda)}([a,b]^2_{<},E_\xi)}\notag\\
\le & c_3\big([x]_{C^\eta([a,b])}\|R^F\|_{\mathscr C^{\eta+\alpha-\lambda}([a,b]^2_<;E_\lambda)}+\|\mathbb X\|_{\mathscr C^{2\eta}([a,b]^2_<)}\|\hat\delta_1\check{f}\|_{\mathscr C^\alpha([a,b]^2_<;E_\omega)}\big)
\label{stima-If-2}
\end{align}
for some positive constant $c_3$, which depends on $\lambda$, $\omega$, $\alpha$, $\eta$, $\xi$ and $b-a$ and stays bounded as $b-a$ tends to zero.
Since $2\eta+\alpha-(\xi\vee\lambda)\ge\theta$, combining \eqref{stima-If-1} and \eqref{stima-If-2}, the assertion follows. 
\end{proof}

\begin{remark}
\label{rmk:un_rough_young_int}
{\rm    
\begin{enumerate}[\rm (i)]
\item 
In fact, the construction of the rough convolution integral can be obtained using the above arguments, also without the extra condition \eqref{cond-x}. In such a new situation, the SG-derivative is not unique and the definition of the convolution integral depends also on the SG-derivative which we choose.
\item 
If $\eta>\frac12$ and $f$ admits SG-derivative $\check f$ with $\hat\delta_1\check f\in \mathscr C^{\zeta}([a,b]^2_<;E)$ and remainder $R^F\in\mathscr C^\rho([a,b]^2_<;E)$ with $\rho>\eta$ and $(\eta+\rho)\wedge (\zeta+2\eta)>1$, then the rough convolution integral $\mathscr I_{Sf}$, defined in \eqref{def_int_conv_rough}, coincides with the Young convolution integral defined by \eqref{new-def}, i.e.,
\begin{align}
\;\;\;\;\;\;\;\mathscr I^{{\it Young}}_{Sf}(s,t)=S(t-s)g_0(s,t)-M_{\delta_{S,2}g_0}(s,t),\qquad\;\,(s,t)\in [a,b]^2_{<},
\label{stelle}
\end{align}
where $g_0$ is defined by \eqref{g0}. Indeed, the integral in \eqref{stelle} is well-defined, since from the definition of SG-derivative it follows that the function $\hat\delta_1f$ belongs to $\mathscr{C}^{\eta}([a,b]^2_{<};E)$, so that $\delta_{S,2}g_0$ belongs to $\mathscr{C}^{2\eta}([a,b]^2_{<};E)$ and $2\eta>1$.

Let us prove that $\mathscr I_{Sf}$ and $\mathscr I^{{\it Young}}_{Sf}$ actually coincide. We need to show that 
\begin{align}
\label{diff_int_rough_young}
S(t-s)\check f(s)\mathbb X(s,t)-M_{\delta_{S,2}g_{\check{f}}}(s,t)+M_{\delta_{S,2}g_0}(s,t)=0
\end{align}
for every $(s,t)\in [a,b]^2_{<}$.
If we set $\mathscr A (s,t):=\check f(s)\mathbb X(s,t)$ for every $(s,t)\in[a,b]^2_<$, then 
$\delta_{S,2}g_{\check{f}}
= \delta_{S,2}g_0 +\delta_{S,2}\mathscr A$ and, arguing as in \eqref{S_incr_A}, we get
\begin{align*}
\;\;\;\;\;\;\;(\delta_{S,2}\mathscr A)(s,t,u)
= & S(t-s)\check f(s)(x(t)-x(s))(x(u)-x(t))
-(\hat\delta_1\check f)(s,t)\mathbb X(t,u)
\end{align*}
for every $(s,t,u)\in [a,b]^3_<$. Therefore, $\delta_{S,2}g_0, \delta_{S,2}\mathscr A \in \mathscr C^{2\eta}([a,b]^2_<;E)$ with $2\eta>1$ and so, from Remark \ref{rmk:lin_M}(i), $M_{\delta_{S,2}g_{\check{f}}}=M_{\delta_{S,2}g_0}+M_{\delta_{S,2}\mathscr A}$. Using this identity and the last part of the statement of Proposition \ref{prop:new_sew_map}, which shows that 
$S(t-s)\mathscr A (s,t)=M_{\delta_{S,2}\mathscr A}(s,t)$
for every $(s,t)\in [a,b]^2_{<}$, we easily deduce that formula \eqref{diff_int_rough_young} is satisfied. 
\end{enumerate}}
\end{remark}

In the following lemma, we exploit some basic properties of the rough convolution integral.

\begin{lemma}
\label{lemma:lin_int}
The following properties are satisfied:
\begin{enumerate}[\rm (i)]
\item 
If $f_1$ and $f_2$ are as in Definition $\ref{def:def_rough_conv_int}$, then $\mathscr I_{S(f_1+f_2)}=\mathscr I_{Sf_1}+\mathscr I_{Sf_2}$;
\item 
Let $f$ be as in Definition $\ref{def:def_rough_conv_int}$.
Then,
\begin{align}
\label{spez_int}
\mathscr I_{Sf}(r,t)=S(t-s)\mathscr I_{Sf}(r,s)+\mathscr I_{Sf}(s,t), \qquad\;\,  (r,s,t)\in[a,b]^3_<.    
\end{align}
\end{enumerate}
\end{lemma}

\begin{proof}
(i) As we have observed in Remark \ref{rmk:lin_T_der}(i), the function $f=f_1+f_2$ admits SG-derivative, which is the sum of the SG-derivative of $f_1$ and $f_2$. It thus follows that
$g_{\check f}=g_{\check f_1}+g_{\check f_2}$ and, consequently, from the linearity of the function $g\mapsto M_g$ (see Remark \ref{rmk:lin_M}) we get
\begin{align*}
\mathscr{I}_{S(f_1+f_2)}(s,t)=&S(t-s)(f_1(s)+f_2(s))(x(t)-x(s))-M_{\delta_{S,2}\check g}(s,t)\\
=&S(t-s)f_1(s)-M_{\delta_{S,2}g_{\check f_1}}(s,t)
+S(t-s)f_2(s)-M_{\delta_{S,2}g_{\check f_2}}(s,t)\\
=&\mathscr{I}_{Sf_1}(s,t)+\mathscr{I}_{Sf_2}(s,t) \end{align*}
for every $(s,t)\in[a,b]^2_<$.

(ii) From the definition of $\mathscr I_{Sf}$ we infer that
\begin{align*}
&\mathscr I_{Sf}(r,t)-\mathscr I_{Sf}(s,t)\\
= & S(t-r)g_{\check f}(r,t)-M_{\delta_{S,2}g_{\check f}}(r,t)-S(t-s)g_{\check f}(s,t)+M_{\delta_{S,2}g_{\check f}}(s,t) \\
= & S(t-s)[S(s-r)g_{\check f}(r,t)-g_{\check f}(s,t)-S(s-r)g_{\check f}(r,s)] \\
& -[M_{\delta_{S,2}g_{\check f}}(r,t)-M_{\delta_{S,2}g_{\check f}}(s,t)+S(t-s)M_{\delta_{S,2}g_{\check f}}(r,s)]\\
& +S(t-s)[S(s-r)g_{\check f}(r,s)-M_{\delta_{S,2}g_{\check f}}(r,s)] \\
= & S(t-s)(\delta_{S,2}g_{\check f})(r,s,t)-(\hat\delta_2M_{\delta_{S,2}g_{\check f}})(r,s,t)+S(t-s)\mathscr I_{Sf}(r,s) \\
= & S(t-s)\mathscr I_{Sf}(r,s) 
\end{align*}
for every $r,s,t\in[a,b]$, with $r\le s\le t$, since from Proposition \ref{prop:new_sew_map} it follows that   
\begin{align*}
(\hat\delta_2M_{\delta_{S,2}g_{\check{f}}})(r,s,t)=S(t-s)(\delta_{S,2}g_{\check{f}})(r,s,t), \qquad\;\, (r,s,t)\in[a,b]^3_<.
\end{align*}
The proof is complete.
\end{proof}

For further use, we need to introduce the rough integral of a ``good'' function with respect to the pair $(x,\mathbb X)$, which in fact coincides with the convolution integral introduced in Definition \ref{def:def_rough_conv_int} when $S(t)=I$ for every $t\geq0$. Indeed, in this case, Hypotheses \ref{hyp-main}, except (iii)-(a), are fulfilled and $\hat\delta_1=\delta_1$, $\hat\delta_2=\delta_2=\delta_{S,2}$. The following proposition collects analogous results to those in  Proposition \ref{prop:new_sew_map} and Lemma \ref{lemma:def_rough_conv_int}.

\begin{proposition}
\label{prop:old_sew_map}
Let Hypotheses $\ref{hyp:x}$ and $\ref{hyp:X_2}$ be satisfied and let $f:[a,b]\to E$ be a continuous function which admits Gubinelli derivative $\check{f}\in C^\zeta([a,b];E_\omega)$ with remainder $\widetilde R^F\in \mathscr C^\rho([a,b]^2_<;E_\lambda)$ and $\mu=(\rho+\eta)\wedge(\zeta+2\eta)>1$.
Then, the rough integral $\mathscr{I}_f$ of the function $f$ with respect to the path $x$ is well-defined by the formula
\begin{align}
\label{def_int_rough}
\mathscr I_{f}(s,t)
= g_{\check f}(s,t)-M_{\delta_2g_{\check f}}(s,t),\qquad\;\,(s,t)\in[a,b]^2_{<},
\end{align}
where $g_{\check f}$ is given by 
\eqref{funct-A}.
Moreover, there exists a positive constant $c_*$,
which depends on $\lambda$, $\omega$, $\rho$, $\eta$, $\zeta$, and $b-a$, such that
\begin{align*}
&\|{\mathscr I}_f-g_{\check f}\|_{\mathscr C^{\mu}([a,b]^2_<;E_{\lambda\wedge\omega})}\\
\leq &  c_*\big([x]_{C^\eta([a,b])}\|\widetilde R^F\|_{\mathscr C^\rho([a,b]^2_<;E_\lambda)}+\|\mathbb X\|_{\mathscr C^{2\eta}([a,b]^2_<)}\|\check{f}\|_{C^\zeta([a,b];E_\omega)}\big).
\end{align*}
\end{proposition}


We conclude this section with the following remarkable result.

\begin{lemma}
\label{lem:ex_S_der_Gf}
Let $x$ be as in Hypotheses $\ref{hyp:x}$,
fix $\alpha\in (0,1)$ and assume that 
$y\in C([a,b];E)\cap B([a,b];E_{\eta+\alpha})$ admits SG-derivative $\check y\in C([a,b];E)\cap B([a,b];E_{\alpha})$ with remainder $R^Y\in \mathscr C^{\eta+\alpha}([a,b]^2_<;E_{\alpha})$. Further, assume that $G\in\mathscr{L}(E)\cap\mathscr{L}(E_{\alpha})\cap\mathscr{L}(E_{\eta+\alpha})$. Then, the function $Gy$ admits SG-derivative and the SG-derivative is $G\check{y}$. Moreover, 
\begin{align}
\label{forma_rem_GY}
R^{GY}(s,t)= [G,\mathfrak a(s,t)]\check{y}(s)
(x(t)-x(s))+[G,\mathfrak a(s,t)]y(s)+GR^Y(s,t)
\end{align}
for every $(s,t)\in[a,b]^2_<$. In particular,
$R^{GY}\in\mathscr C^{\eta+\alpha}([a,b]^2_<;E)$ and 
\begin{align}
\|R^{GY}\|_{\mathscr C^{\eta+\alpha}([a,b]^2_<;E)}\leq & C_{\alpha,0}\mathfrak{g}_{0,\alpha}[x]_{C^\eta([a,b])}\|\check{y}\|_{B([a,b];E_{\alpha})}\notag \\
& + C_{\eta+\alpha,0}\mathfrak{g}_{0,\eta+\alpha}\|y\|_{B([a,b];E_{\eta+\alpha})}\notag \\
& + K_{\alpha,0}\mathfrak{g}_0\|R^Y\|_{\mathscr C^{\eta+\alpha}([a,b]^2_<;E_{\alpha})},
\label{stima_S_der_Gy}
\end{align}
where $C_{\gamma,\delta}=C_{\gamma,\delta,b-a}$.
\end{lemma}

\begin{proof}
Adding and subtracting $G\mathfrak a(s,t)y(s)$ and then $S(t-s)G\check y(s)(x(t)-x(s))$ from $(\hat\delta_1 Gy)(s,t)$, we get
\begin{align}
(\hat\delta_1Gy)(s,t)
= & Gy(t)-Gy(s)-\mathfrak a(s,t)G(y(s)) \notag \\
= & (G\hat\delta_1y)(s,t)+[G,\mathfrak a(s,t)]y(s) \notag \\
= & GS(t-s)\check{y}(s)(x(t)-x(s))+GR^Y(s,t)+[G,\mathfrak a(s,t)]y(s) \notag \\
= & S(t-s)G\check{y}(s)(x(t)-x(s))+[G,\mathfrak a(s,t)]\check{y}(s)(x(t)-x(s)) \notag \\
& +[G,\mathfrak a(s,t)]y(s)+GR^Y(s,t)\notag\\
=&
S(t-s)G\check{y}(s)(x(t)-x(s))+R^{GY}(s,t)
\label{delta_1_G}
\end{align}
for every $(s,t)\in[a,b]^2_<$, where we have used the fact that $[A,B]=[A-I,B]=[A,B-I]$ for every $A,B\in \mathscr{L}(E)$.

From the assumptions on $G$, $y$, $\check y$ and $R^Y$, for every $(s,t)\in[a,b]^2_<$ we can estimate
\begin{align}
\notag 
&|[G,\mathfrak a(s,t)]\check{y}(s)(x(t)-x(s))|_E\le  C_{\alpha,0}\mathfrak{g}_{0,\alpha}[x]_{C^{\eta}([a,b])}\|\check{y}\|_{B([a,b];E_{\alpha})}|t-s|^{\eta+\alpha}, \\[1mm]
& |[G,\mathfrak a(s,t)]y(s)|_{E}\leq  C_{\eta+\alpha,0}\mathfrak{g}_{0,\eta+\alpha}\|y\|_{B([a,b];E_{\eta+\alpha})}|t-s|^{\eta+\alpha}, \notag\\[1mm] 
& |GR^Y(s,t)|_E
\leq  K_{\alpha,0}\mathfrak{g}_0\|R^Y\|_{\mathscr C^{\eta+\alpha}([a,b]^2_<;E_{\alpha})}|t-s|^{\eta+\alpha}.
\label{stima_comm_y}
\end{align}

From \eqref{stima_comm_y} it follows that
$R^{GY}\in\mathscr{C}^{\rho+\eta}([a,b]^2;E)$ and satisfies estimate \eqref{stima_S_der_Gy}. Hence, \eqref{delta_1_G} allows us to conclude that
$Gy$ admits SG-derivative, the SG-derivative is $G\check{y}$ and the remainder is $R^{GY}$.
\end{proof}

\begin{remark}
{\rm The results in Lemma \ref{lem:ex_S_der_Gf} are not optimal when $G$
commutes with the semigroup $(S(t))_{t\ge 0}$. Indeed, in this case
$R^{GY}$ reduces to $GR^Y$. Hence, $Gy$ admits SG-derivative (if $y$ does), without any extra assumption on $y$, $\check{y}$ and $R^Y$.}   
\end{remark}

\section{Rough differential equations in infinite dimension}
\label{sect-3}
In this section we deal with the Cauchy problem
\begin{align}
\label{cauchy_prob_rough}
\left\{
\begin{array}{ll}
dy(t)=Ay(t)dt+Gy(t)dx(t),    &  t\in(0,T], \vspace{1mm} \\
y(0)=\psi,  
\end{array}
\right.
\end{align}
where $T>0$, $G\in\mathscr{L}(E)$ and $\psi\in E$. We look for a mild solution to \eqref{cauchy_prob_rough}, that is, roughly speaking, a function $y:[0,T]\to E$ such that $\mathscr I_{SGy}(0,\cdot)$ is well-defined in $[0,T]$ and $y$ satisfies the equation
\begin{align*}
y(t)=S(t)\psi+\mathscr I_{SGy}(0,t), \qquad\;\, t\in[0,T].
\end{align*}
The precise definition of mild solution will be given later. 

Our aim consists in proving existence and uniqueness of a mild solution $y$ to problem \eqref{cauchy_prob_rough} by means of a fixed-point argument. Moreover, we will show that such a solution takes value in a smaller space, that is, $y(t)\in D(A)$ for every $t\in(0,T]$.

Throughout this section, besides Hypotheses \ref{hyp-main}, we assume the following conditions on $x$, $\mathbb{X}$ and $G$.

\begin{hypotheses}
\label{hyp-3.1}
\begin{enumerate}[\rm (i)]
\item 
$x\in C^{\eta}([0,T])$ for some $\eta\in \left (\frac{1}{3},\frac{1}{2}\right ]$ and there exists $\alpha\in (1-2\eta,\eta)$ such that
\begin{eqnarray*}
\limsup_{t\to s^+}\frac{|x(t)-x(s)|}{|t-s|^{\eta+\alpha}}=\infty,\qquad\;\,s\in [0,T);
\end{eqnarray*}
\item 
$\mathbb{X}\in \mathscr C^{2\eta}([0,T]^2_{<})$ satisfies condition \eqref{form-diff-X};
\item 
$G\in\mathscr{L}(E)\cap\mathscr{L}(E_{\alpha})\cap\mathscr{L}(E_{\eta+\alpha})$.
\end{enumerate}
\end{hypotheses}

Now, we define the space where we look for a solution to \eqref{cauchy_prob_rough}.

\begin{definition}
For every $a,b\in [0,T]$, with $a<b$, we introduce the space
\begin{align*}
{\mathscr Y}_{\alpha,\eta}(a,b)
=\{& y\in C((a,b];E_{\eta+\alpha})\cap B([a,b];E_{\eta+\alpha})\textrm{ which admit SG-derivative } \\
& \textrm{}\check{y}\in C((a,b];E_\alpha) \cap  B([a,b];E_\alpha)\textrm{ with respect to $x$ with remainder }\\
&  R^Y\in \mathscr C^{\eta+\alpha}([a,b]^2_<;E_\alpha) \textrm{and } \hat\delta_1\check{y}\in \mathscr C^\alpha([a,b]^2_<;E_\alpha)\}.
\end{align*}
${\mathscr Y}_{\alpha,\eta}(a,b)$ is a Banach space when endowed with the norm
\begin{align*}
\|y\|_{\mathscr Y_{\alpha,\eta}(a,b)}
= & \|y\|_{B([a,b];E_{\eta+\alpha})}+ \|\check{y}\|_{B([a,b];E_\alpha)}+\|R^Y\|_{\mathscr C^{\eta+\alpha}([a,b]^2_<;E_\alpha)} \\
& +\|\hat\delta_{1}\check{y}\|_{\mathscr C^\alpha([a,b]^2_<;E_\alpha)}
\end{align*}
for every 
$y\in {\mathscr Y}_{\alpha,\eta}(a,b)$. 
\end{definition}

As the following remark shows, the elements of $\mathscr{Y}_{\alpha,\eta}(a,b)$ enjoy some additional regularity.

\begin{remark}
\label{rmk:reg_hat_delta_1_y}
{\rm If $y\in\mathscr Y_{\alpha,\eta}(a,b)$, then $\hat\delta_1y\in \mathscr C^{\eta}([a,b]^2_<;E_\alpha)$. Indeed, for every $(s,t)\in[a,b]^2_<$, we get
\begin{align*}
|(\hat\delta_1 y)(s,t)|_{E_\alpha}
\leq  & L_{\alpha,\alpha}\|\check y\|_{B([a,b];E_{\alpha})}[x]_{C^{\eta}([a,b])}|t-s|^{\eta}+\|R^Y\|_{\mathscr C^{\eta+\alpha}([a,b]^2_<;E_\alpha)}|t-s|^{\eta+\alpha},
\end{align*}
which gives 
$\|\hat\delta_1y\|_{\mathscr C^{\eta}([a,b]^2_<;E_\alpha)}\leq (L_{\alpha,\alpha}[x]_{C^{\eta}([a,b])}+(b-a)^{\alpha})\|y\|_{\mathscr Y_{\alpha,\eta}(a,b)}$.

As a byproduct, $y\in C^{\eta}([a,b];E_{\alpha})$. Indeed,
\begin{align*}
|y(t)-y(s)|_{E_{\alpha}}\le &|(\hat\delta_1y)(s,t)|_{E_{\alpha}}+|\mathfrak{a}(s,t)y(s)|_{E_{\alpha}}\\
\le &(\|\hat\delta_1y\|_{{\mathscr C}^{\eta}([a,b];E_{\alpha})}
+C_{\eta+\alpha,\alpha,b-a}\|y\|_{B([a,b];E_{\eta+\alpha})})|t-s|^{\eta}
\end{align*}
for every $(s,t)\in [a,b]^2_{<}$.

Using a similar argument, based on the fact that 
$\hat\delta_1\check y\in\mathscr{C}^{\alpha}([a,b]^2_{<};E_{\alpha})$ and $\check y\in B([a,b];E_{\alpha})$, it can be proved that $\check y$ belongs to $C^{\alpha}([a,b];E)$ and
\begin{eqnarray*}
|y(t)-y(s)|_E\le \big (K_{\alpha,0}\|\hat\delta_1\check y\|_{\mathscr{C}^{\alpha}([a,b]^2_{<};E_{\alpha})}
+C_{\alpha,0,b-a}\|\check y\|_{B([a,b];E_{\alpha})}\big )|t-s|^{\alpha}
\end{eqnarray*}
for every $(s,t)\in [a,b]^2_{<}$.
}
\end{remark}

On $\mathscr{Y}_{\alpha,\eta}(0,T)$ we introduce the operator $\Gamma$, defined as  
\begin{align}
\label{def_gamma_op}    
\Gamma(y)=S(\cdot)\psi+\mathscr I_{SGy}(0,\cdot), \qquad\;\, y\in {\mathscr Y}_{\alpha,\eta}(0,T).
\end{align}

Now we are able to give the precise definition of mild solution to \eqref{cauchy_prob_rough}.
\begin{definition}
We say that $y\in {\mathscr Y}_{\alpha,\eta}(0,T)$ is a mild solution to \eqref{cauchy_prob_rough} if  $y=S(\cdot)\psi+\mathscr I_{SGy}(0,\cdot)$.
\end{definition}

\begin{remark}
{\rm If $y$ is a mild solution to \eqref{cauchy_prob_rough}, then from \eqref{def_int_conv_rough}, it follows that it solves the equation
\begin{align}
y(t)=S(t)\psi+S(t)(G\psi)(x(t)-x(0))+S(t)(G^2\psi)\mathbb X(0,t)-M_{\delta_{S,2}g_{G^2y}}(0,t)
\label{mild_sol_form_espl}
\end{align}
for every $t\in[0,T]$. The dependence on $y$ in the right-hand side of \eqref{mild_sol_form_espl} is implicit and lies in the last term.}
\end{remark}

\begin{proposition}
\label{prop:Gamma_mappa_Y_in_Y}
For every $\psi\in E_{\eta+\alpha}$, the operator $\Gamma$ maps ${\mathscr Y}_{\alpha,\eta}(0,T)$ into itself for every $T>0$. Moreover, there exists a positive constant $\mathfrak{R}$,  which depends only on $T$, $\eta$, $\alpha$, $[x]_{C^{\eta}([0,T])}$, $\|\mathbb{X}\|_{\mathscr C^{2\eta}([0,T]^2_{<})}$, $G$ and the constants appearing in Hypotheses $\ref{hyp-main}$, such that, if $y$ is a fixed point of the operator $\Gamma$, then
\begin{align}
\label{stima_a_priori_Y}    
\|y\|_{{\mathscr Y}_{\alpha,\eta}(0,T)}\leq \mathfrak R |\psi|_{E_{\eta+\alpha}}.
\end{align}
\end{proposition}

\begin{proof}
For the reader's convenience, we split the proof into different steps.
In the first three steps, we prove that $\Gamma$ maps $\mathscr{Y}_{\alpha,\eta}(0,T)$ into itself. In the last one, we complete the proof, checking estimate \eqref{stima_a_priori_Y}.

Throughout the proof, $T>0$ is fixed.
For further use, in the estimate that we prove, we stress the dependence on the H\"older norms of $x$ and $\mathbb{X}$. Hence, 
by $\mathfrak{C}_j$ we denote positive constants, 
which may depend on $\alpha$, $\eta$, $T$, $G$, the constants in \eqref{stime_smgr} as well as the embedding constants in Hypothesis \ref{hyp-main}(ii), but are independent of $x$, $\mathbb{X}$ and $y$.

\textit{Step 1.} Here, we prove that 
the rough convolution integral $\mathscr I_{SGy}$ is well-defined for every $y\in\mathscr{Y}_{\alpha,\eta}(0,T)$. 

We fix $y\in {\mathscr Y}_{\alpha,\eta}(0,T)$ and use Lemma \ref{lem:ex_S_der_Gf} to infer that $Gy$ admits SG-derivative $G\check{y}$, with respect to $x$ with remainder $R^{GY}\in \mathscr C^{\eta+\alpha}([0,T]^2_<;E)$. Further, observing that
$(\hat\delta_1 G\check{y})(s,t)=G(\hat\delta_1\check y)(t,s)+
[G,\mathfrak a(s,t)]\check{y}(s)$ for every $(s,t)\in [0,T]^2_{<}$ and recalling that $\check y\in B([0,T];E_\alpha)$ and $\hat\delta_1\check y\in\mathscr C^\alpha([0,T]^2_{<};E_\alpha)$, 
it follows easily that $\hat\delta_1 G\check{y}\in \mathscr{C}^{\alpha}([0,T]^2_{<};E)$ and
\begin{align}
\|\hat\delta_1 G\check{y}\|_{\mathscr{C}^{\alpha}([0,T]^2_{<};E)}
\leq & K_{\alpha,0}\mathfrak{g}_0\|\hat\delta_1 \check{y}\|_{\mathscr C^\alpha([0,T]^2_<;E_\alpha)} +C_{\alpha,0}\mathfrak{g}_{0,\alpha}\|\check{y}\|_{B([0,T];E_{\alpha})}.
\label{stima_delta_1Gchecky}
\end{align}

Hence, from Lemma \ref{lem:reg_int_conv_rough}, with $\lambda=\omega=0$ and $\omega_1=\alpha$,
estimates \eqref{stima_S_der_Gy} and \eqref{stima_delta_1Gchecky},
it follows that the rough convolution integral $\mathscr I_{SGy}$ is well-defined, it belongs to  
${\mathscr C}^{\theta}([0,T]^2_<;E_\xi)$,  for every $\xi\in[0,1)$, where 
$\theta=\eta-(\xi-\eta-\alpha)^+$,
and
\begin{align}
\|\mathscr I_{SGy}\|_{\mathscr C^{\theta}([0,T]^2_<;E_\xi)}\le &\mathfrak{C}_0\Big (\mathfrak{g}_{\eta+\alpha}\|y\|_{B([0,T];E_{\eta+\alpha})}+\mathfrak{g}_{\alpha}
\|\mathbb X\|_{\mathscr C^{2\eta}([0,T]_<^2)}\|\check y\|_{B([0,T];E_{\alpha})}\notag\\
&\quad\;\,+
C_{\alpha,0}(\mathfrak{g}_0+\mathfrak{g}_{\alpha})([x]_{C^{\eta}([0,T])}^2+\|\mathbb{X}\|_{\mathscr{C}^{2\eta}([0,T]^2_{<})})\|\check{y}\|_{B([0,T];E_{\alpha})}\notag \\
&\quad\;\, + C_{\eta+\alpha,0}(\mathfrak{g}_0+\mathfrak{g}_{\eta+\alpha})[x]_{C^{\eta}([0,T])}\|y\|_{B([0,T];E_{\eta+\alpha})}
\notag \\
&\quad\;\, + K_{\alpha,0}\mathfrak{g}_0[x]_{C^{\eta}([0,T])}\|R^Y\|_{\mathscr C^{\eta+\alpha}([0,T]^2_<;E_{\alpha})}\notag\\
&\quad\;\,+K_{\alpha,0}\mathfrak{g}_0\|\mathbb X\|_{\mathscr C^{2\eta}([0,T]^2_<)}\|\hat\delta_1\check y\|_{\mathscr{C}^{\alpha}([0,T]^2_{<};E_{\alpha})}\Big )\notag\\
\le & \mathfrak{C}_1(1+[x]^2_{C^{\eta}([0,T])}
+\|\mathbb{X}\|_{\mathscr C^{2\eta}([0,T]^2_{<})})\|y\|_{{\mathscr Y}_{\alpha,\eta}(0,T)}
\label{stima_int_completa-1}
\end{align}
for some positive constants $\mathfrak{C}_0$ and $\mathfrak{C}_1$, which depend also on $\xi$.

\textit{Step 2.} Here, we prove that $\Gamma(y)\in C((0,T];E_{\eta+\alpha})\cap B([0,T];E_{\eta+\alpha})$. 

We fix $y\in {\mathscr Y}_{\alpha,\eta}(0,T)$ and compute the $E_{\eta+\alpha}$-norm of $S(t)\psi+\mathscr I_{SGy}(0,t)$ for every $t\in[0,T]$. Taking $\xi=\eta+\alpha$, we obtain that
\begin{align*}
|S(t)\psi+\mathscr I_{SGy}(0,t)|_{E_{\eta+\alpha}}
\leq & L_{\eta+\alpha,\eta+\alpha}|\psi|_{E_{\eta+\alpha}}+\|\mathscr I_{SGy}\|_{\mathscr C^{\eta}([0,T]^2_<;E_{\eta+\alpha})}T^{\eta}
\end{align*}
for every $t\in[0,T]$. 
It follows that $\Gamma(y)\in B([0,T];E_{\eta+\alpha})$.

Let us study the continuity of $\Gamma(y)$ in $(0,T]$ with respect to the $E_{\eta+\alpha}$-norm. From Remark \ref{strong-cont-smgr}, it follows that the function $S(\cdot)\psi$ belongs to $C((0,T];E_{\eta+\alpha})$. Moreover, from Lemma \ref{lemma:lin_int}, we know that
\begin{align*}
\mathscr I_{SGy}(0,t)-\mathscr I_{SGy}(0,s)
= \mathscr I_{SGy}(s,t)+\mathfrak a(s,t)\mathscr I_{SGy}(0,s)
\end{align*}
for every $(s,t)\in[0,T]^2_<$ (let us notice that if $s=0$ the computations simplify, since $\mathscr I_{SGy}(0,0)=0$). If we fix $\varepsilon\in(0,1-\eta-\alpha]$, then from \eqref{stime_smgr}(b) and taking  $\xi=\eta+\alpha$ and  $\xi=\eta+\alpha+\varepsilon$ in \eqref{stima_int_completa-1}, we get
\begin{align*}
&|\mathscr I_{SGy}(0,t)-\mathscr I_{SGy}(0,s)|_{E_{\eta+\alpha}}\\
\leq & \|\mathscr I_{SGy}\|_{\mathscr C^{\eta}([0,T]^2_<;E_{\eta+\alpha})}|t-s|^{\eta}\\
&+C_{\eta+\alpha+\varepsilon,\eta+\alpha}\|\mathscr I_{SGy}\|_{\mathscr C^{\eta-\varepsilon}([0,T]^2_<;E_{\eta+\alpha+\varepsilon})}T^{\eta-\varepsilon}|t-s|^{\varepsilon}
\end{align*}
for every $(s,t)\in[0,T]^2_<$. This proves the continuity of $\mathscr I_{SGy}(0,\cdot)$ in $[0,T]$ with values in $E_{\eta+\alpha}$. Hence, $\Gamma(y)\in C((0,T];E_{\eta+\alpha})\cap B([0,T];E_{\eta+\alpha})$.


\textit{Step 3.} Here, we complete the proof, showing that $\Gamma(y)$ admits SG-derivative and analyze its smoothness. 

From \eqref{spez_int} and the fact that $\hat\delta_1S(\cdot)\psi=0$, we infer that
\begin{align*}
(\hat\delta_1\Gamma(y))(s,t)
= & \mathscr I_{SGy}(s,t)
\\
= & S(t-s)Gy(s)(x(t)-x(s)) +S(t-s)G\check{y}(s)\mathbb X(s,t)
-M_{\delta_{S,2}g_{G\check{y}}}(s,t)
\end{align*}
for every $(s,t)\in[0,T]^2_<$. We claim that $\Gamma(y)$ admits SG-derivative $Gy$ with remainder 
\begin{align}
R^{\Gamma(y)}(s,t)=S(t-s)G\check{y}(s)\mathbb{X}(s,t)-M_{\delta_{S,2}g_{G\check{y}}}(s,t), \qquad\;\, (s,t)\in[0,T]^2_<.
\label{rem-gamma}
\end{align}
To prove the claim, we show that $R^{\Gamma(y)}\in\mathscr C^{2\eta}([0,T]^2_<;E_\alpha)$. From the assumptions on $G$ and $\check{y}$, and \eqref{stima_M}, with $F=GY$, $\varepsilon=\alpha$ and $\lambda=\omega=0$, we infer that
\begin{align}
|R^{\Gamma(y)}(s,t)|_{E_\alpha}
\leq & L_{\alpha,\alpha}\mathfrak{g}_{\alpha}\|\check{y}\|_{B([0,T];E_\alpha)}\|\mathbb X\|_{\mathscr C^{2\eta}([0,T]^2_<)}|t-s|^{2\eta}\notag \\
& +c_2\big([x]_{C^\eta([0,T])}\|R^{GY}\|_{\mathscr C^{\eta+\alpha}([0,T]^2_<;E)} \notag \\
&\qquad\quad+\|\mathbb X\|_{\mathscr C^{2\eta}([0,T]^2_<)}\|\hat\delta_1G\check{y}\|_{\mathscr C^\alpha([0,T]^2_<;E)}\big)|t-s|^{2\eta}
\label{star-star}
\end{align}
for every $(s,t)\in[0,T]^2_<$. The claim is proved. Moreover, from \eqref{star-star} and
\eqref{stima_S_der_Gy} it follows that
\begin{align}
\|R^{\Gamma(y)}\|_{\mathscr{C}^{2\eta}([0,T]^2_<;E_\alpha)}
\le & \mathfrak{C}_2(1+[x]_{C^{\eta}([0,T])}^2+\|\mathbb{X}\|_{\mathscr C^{2\eta}([0,T]^2_{<})})\|y\|_{\mathscr{Y}_{\alpha,\eta}(0,T)}.
\label{lilly}
\end{align}

The fact that $Gy\in B([0,T];E_\alpha)\cap C((0,T];E_\alpha)$ follows from the boundedness of $G$ in $E_{\alpha}$ and the regularity of $y$. To show that $\hat\delta_1 Gy$ belongs to $ \mathscr C^\alpha([0,T]^2_<;E_\alpha)$ we observe that
\begin{align*}
|(\hat\delta_1Gy)(s,t)|_{E_\alpha}
= & |Gy(t)-Gy(s)-S(t-s)Gy(s)|_{E_\alpha} \notag \\
\leq & |G(\hat\delta_1 y)(s,t)|_{E_\alpha}+|[G,S(t-s)]y(s)|_{E_\alpha} \notag \\
\leq & \mathfrak{g}_{\alpha}
|S(t-s)\check{y}(s)(x(t)-x(s))+R^Y(s,t)|_{E_\alpha}
+|[G,\mathfrak a(s,t)]y(s)|_{E_\alpha}\notag\\
\leq & \mathfrak{g}_{\alpha}L_{\alpha,\alpha}\|\check{y}\|_{B([0,T];E_\alpha)}[x]_{C^\eta([0,T])}|t-s|^\eta\notag \\
& +\mathfrak{g}_{\alpha}\|R^Y\|_{\mathscr C^{\eta+\alpha}([0,T]^2_<;E_\alpha)}|t-s|^{\eta+\alpha}\notag\\
&+\mathfrak{g}_{\alpha,\eta+\alpha}C_{\eta+\alpha,\alpha}\|y\|_{B([0,T];E_{\eta+\alpha})}|t-s|^{\eta}
\end{align*}
for every $(s,t)\in[0,T]^2_<$. Since $\eta>\alpha$, we conclude that $\hat\delta_1 Gy\in\mathscr C^\alpha([0,T]^2_<;E_\alpha)$ and
\begin{align}
\|\hat\delta_1Gy\|_{\mathscr{C}^{\alpha}([0,T]^2_{<};E_{\alpha})}\le
\mathfrak{C}_3(1+[x]_{C^{\eta}([0,T])})\|y\|_{\mathscr{Y}_{\alpha,\eta}(0,T)}.
\label{confine}
\end{align}

Summing up, we have proved that $\Gamma(y)\in {\mathscr Y}_{\alpha,\eta}(0,T)$. 

{\em Step 4.}
We fix $[a,b]\subset [0,T]$ and prove that there exists a positive constant $\mathfrak{C}_4$, such that
\begin{align}
\|\Gamma(y)\|_{\mathscr Y_{\alpha,\eta}(a,b)}
\leq & (|(\Gamma(y))(a)|_{E_{\eta+\alpha}}+K_{\eta+\alpha,\alpha}\mathfrak{g}_{\eta+\alpha}|y(a)|_{E_{\eta+\alpha}})L_{\eta+\alpha,\eta+\alpha}\notag\\
& +\mathfrak{C}_4(1+[x]_{C^{\eta}([0,T])}+\|\mathbb{X}\|_{\mathscr C^{2\eta}([0,T]^2_{<})})\|y\|_{{\mathscr Y}_{\alpha,\eta}(a,b)}(b-a)^{\eta-\alpha},
\label{stima_parz_mild_sol}
\end{align}
where, here and below, $L_{\gamma,\delta}=L_{\gamma,\delta,T}$ and
$C_{\gamma,\delta}=C_{\gamma,\delta,T}$.
From this estimate, it follows that, if $y$ is a fixed point of the operator $\Gamma$ and $(b-a)^{\eta-\alpha}\leq \frac12(\mathfrak C_4(1+[x]_{C^{\eta}([0,T])}+\|\mathbb{X}\|_{\mathscr C^{2\eta}([0,T]^2_{<})}))^
{-1}$, then
\begin{align}
\|y\|_{\mathscr Y_{\alpha,\eta}(a,b)}
\leq 2(1+K_{\eta+\alpha,\alpha}\mathfrak{g}_{\eta+\alpha})L_{\eta+\alpha,\eta+\alpha}|y(a)|_{E_{\eta+\alpha}}. 
\label{stima_a_priori_fin_1}
\end{align}

Let us prove \eqref{stima_parz_mild_sol}. In view of the previous steps of the proof, we just need to estimate the $B([a,b];E_{\eta+\alpha})$-norm of $\Gamma(y)$ and the $B([a,b];E_{\alpha})$-norm of $Gy$ (its SG-derivative).
Note that
\eqref{spez_int} implies that 
$\Gamma(y)=S(\cdot-a)\Gamma(y)(a)+\mathscr I_{SGy}(a,\cdot)$ in $[a,b]$.  
Hence, using \eqref{stime_smgr}(a) and \eqref{stima_int_completa-1} we can estimate
\begin{align}
|\Gamma(y)(t)|_{E_{\eta+\alpha}}
\leq &  L_{\eta+\alpha,\eta+\alpha}|(\Gamma(y))(a)|_{E_{\eta+\alpha}}+|\mathscr I_{SGy}(a,t)|_{E_{\eta+\alpha}} \notag \\
\leq & L_{\eta+\alpha,\eta+\alpha}|(\Gamma(y))(a)|_{E_{\eta+\alpha}} \notag \\
& +\mathfrak C_1(1+[x]_{C^{\eta}([0,T])}^2+
\|\mathbb{X}\|_{\mathscr C^{2\eta}([0,T]^2_{<})})\|y\|_{{\mathscr Y}_{\alpha,\eta}(a,b)}(b-a)^{\eta}
\label{stima_a_priori_y}
\end{align}
for every $t\in[a,b]$. Further, since
\begin{align*}
Gy(t)=&G(\hat\delta_1y)(a,t)
+GS(t-a)y(a)\\
=&
(GS(t-a)\check y(a))(x(t)-x(a))
+GR^Y(a,t)+GS(t-a)y(a)
\end{align*}
for every $t\in [a,b]$, it follows that
\begin{align}
|Gy(t)|_{E_\alpha}
\leq &  
|GS(t-a)\check y(a)|_{E_{\alpha}}|x(t)-x(a)|
+|GR^Y(a,t)|_{E_{\alpha}}+|GS(t-a)y(a)|_{E_{\alpha}}\notag\\
\le &
\mathfrak{g}_{\alpha}L_{\alpha,\alpha}\|\check y\|_{B([a,b];E_{\alpha})}[x]_{C^{\eta}([0,T])}(b-a)^{\eta}\notag\\
&+\mathfrak{g}_{\alpha}\|R^Y\|_{\mathscr{C}^{\eta+\alpha}([a,b]^2_{<};E_{\alpha})}(b-a)^{\eta+\alpha}
+\mathfrak{g}_{\eta+\alpha}K_{\eta+\alpha,\alpha}L_{\eta+\alpha,\eta+\alpha}|y(a)|_{E_{\eta+\alpha}}\notag\\
\le &
\mathfrak{C}_5(1+[x]_{C^{\eta}([0,T])}+
\|\mathbb{X}\|_{\mathscr C^{2\eta}([0,T]^2_{<})})\|y\|_{{\mathscr Y}_{\alpha,\eta}(a,b)}(b-a)^{\eta} \notag \\
& + K_{\eta+\alpha,\alpha}\|G\|_{\mathscr{L}(E_{\eta+\alpha})}L_{\eta+\alpha,\eta+\alpha}|y(a)|_{E_{\eta+\alpha}}
\label{stima_a_priori_S_der_sup}
\end{align}
for every $t\in [a,b]$.

From \eqref{lilly}, \eqref{confine}, \eqref{stima_a_priori_y} and \eqref{stima_a_priori_S_der_sup}, estimate \eqref{stima_parz_mild_sol} follows, with $\mathfrak C_4=\mathfrak (\mathfrak{C}_1+2\mathfrak{C}_5)T^{\eta-\alpha}+\mathfrak C_2+2\mathfrak C_3$. 

{\it Step 5}. Now, we extend estimate \eqref{stima_a_priori_fin_1} to the whole interval $[0,T]$. For this purpose, we set
\begin{align}
\overline T
:= & 
(2\mathfrak C_4(1+[x]_{C^{\eta}([0,T])}^2+\|\mathbb{X}\|_{\mathscr C^{2\eta}([0,T]^2_{<})}))^{-\frac{1}{\eta-\alpha}}
\label{choice_ovT}
\end{align}
and consider a partition $\{t_0=0<t_1<\ldots<t_N=T\}$ of $[0,T]$ such that $t_n-t_{n-1}\leq \overline T$ for every $n=1,\ldots,N$. Further, we set 
\begin{align*}
\Phi(r)=2(1+K_{\eta+\alpha,\alpha}\mathfrak{g}_{\eta+\alpha}) )L_{\eta+\alpha,\eta+\alpha} r
,\qquad\;\,r\in [0,\infty).
\end{align*}
Since $\Phi$ is an increasing function, from \eqref{stima_a_priori_fin_1}, with $[a,b]$ replaced by $[t_{n-1},t_n]$, it follows that
\begin{align*}
\|y\|_{{\mathscr Y}_{\alpha,\eta}(t_{n-1},t_n)}\leq \Phi(|y(t_{n-1})|_{E_{\eta+\alpha}})\leq \Phi^n(|\psi|_{E_{\eta+\alpha}}),\qquad\;\,n=1,\ldots,N.
\end{align*}
This implies that
\begin{align}
\label{stima_a_priori_y_checky}
\|y\|_{B([0,T];E_{\eta+\alpha})}
+\|\check y\|_{B([0,T];E_\alpha)}\leq \Phi^N(|\psi|_{E_{\eta+\alpha}})
\end{align}
and, consequently,
from \eqref{stima_a_priori_fin_1} and \eqref{stima_a_priori_y_checky} we get
\begin{align}
\label{stima_a_priori_s_t}
\|y\|_{{\mathscr Y}_{\alpha,\eta}(s,t)}\leq 2(1+K_{\eta+\alpha,\alpha}\mathfrak{g}_{\eta+\alpha} )L_{\eta+\alpha,\eta+\alpha}|y(s)|_{E_{\eta+\alpha}}
\leq \Phi^{N+1}(|\psi|_{E_{\eta+\alpha}})  
\end{align}
for every $(s,t)\in[0,T]^2_<$ with $t-s\leq \overline T$.

Using \eqref{stima_a_priori_s_t}, we can complete the estimate
of the  $\mathscr{Y}_{\alpha,\eta}(0,T)$-norm of $y$. In view of \eqref{stima_a_priori_y_checky}, we need to consider the functions $R^{Y}$ and $\hat\delta_1\check{y}$.

Clearly, if $(s,t)\in [0,T]^2_<$ is such that $t-s\leq \overline T$, then from \eqref{stima_a_priori_s_t} it follows that
\begin{align*}
&\frac{|R^Y(s,t)|_{E_\alpha}}{|t-s|^{\eta+\alpha}}\leq \Phi^{N+1}(|\psi|_{E_{\eta+\alpha}}), \qquad\;\,   
\frac{|(\hat\delta_1\check y)(s,t)|_{E_\alpha}}{|t-s|^{\alpha}}
\leq \Phi^{N+1}(|\psi|_{E_{\eta+\alpha}}).
\end{align*}

To estimate the previous ratios when $t-s>\overline T$, we recall that $R^Y(s,t)=(\hat\delta_1y)(s,t)-S(t-s)\check y(s)(x(t)-x(s))$ and that $(\hat \delta_1y)(s,t)=y(t)-S(t-s)y(s)$ for every $(s,t)\in[0,T]^2_<$. Therefore, 
\begin{align*}
\frac{|R^Y(s,t)|_{E_\alpha}}{|t-s|^{\eta+\alpha}}
\leq &  \frac{|R^Y(s,t)|_{E_\alpha}}{\overline{T}^{\eta+\alpha}}\\
\le &
\frac{(1+L_{\alpha,\alpha})K_{\eta+\alpha,\alpha}}{\overline T^{\eta+\alpha}}\|y\|_{B([0,T];E_{\eta+\alpha})} +\frac{L_{\alpha,\alpha}}{\overline T^{\alpha}}[x]_{C^\eta([0,T])}\|\check y\|_{B([0,T];E_\alpha)}\\
\le &
\bigg (\frac{(1+L_{\alpha,\alpha})K_{\eta+\alpha,\alpha}}{\overline{T}^{\eta+\alpha}}+\frac{ L_{\alpha,\alpha}}{\overline T^{\eta}}[x]_{C^\eta([0,T])}\bigg )\Phi^N(|\psi|_{E_{\eta+\alpha}})\\
=&\!:\mathfrak{C}_7\Phi^N(|\psi|_{E_{\eta+\alpha}})
\end{align*}
for every $(s,t)\in[0,T]^2_<$, with $t-s>\overline T$. 

Similarly, since $(\hat\delta_1\check y)(s,t)=\check y(t)-S(t-s)\check y(s)$ for every $(s,t)\in[0,T]^2_<$, from 
\eqref{stima_a_priori_y_checky}, it follows that
\begin{align*}
\frac{|(\hat\delta_1\check y)(s,t)|_{E_\alpha}}{|t-s|^\alpha}
\leq \frac{(1+L_{\alpha,\alpha})\Phi^N(|\psi|_{\eta+\alpha})}{\overline T^\alpha}=:\mathfrak{C}_8\Phi^N(|\psi|_{E_{\eta+\alpha}})
\end{align*}
if $(t,s)\in [0,T]^2$ and $t-s>\overline T$.
Hence,
\begin{align}
\label{stima_a_priori_R^Y_fin}
\|R^Y\|_{\mathscr C^{\eta+\alpha}([0,T]^2_<;E_\alpha)}
\leq \max\{1,\mathfrak{C}_7\}\Phi^{N+1}(|\psi|_{E_{\eta+\alpha}})
\end{align}
and
\begin{align}
\label{stima_a_priori_checky_fin}
\|\hat\delta_1 \check y\|_{\mathscr C^\alpha([0,T]^2_<;E_\alpha)}
\leq \max\{1, \mathfrak{C}_8\}\Phi^{N+1}(|\psi|_{E_{\eta+\alpha}}),
\end{align}
where we took into account that $L_{\eta+\alpha,\eta+\alpha}\ge 1$.
From \eqref{stima_a_priori_y_checky}, \eqref{stima_a_priori_R^Y_fin} and \eqref{stima_a_priori_checky_fin} we get estimate \eqref{stima_a_priori_Y} with
$\mathfrak{R}=(1+\max\{1,\mathfrak{C}_7\}+\max\{1,\mathfrak{C_8}\})\Phi^{N+1}(1)$.
\end{proof}

\begin{remark}
\label{rem-utile}
{\rm The proof of Proposition \ref{prop:Gamma_mappa_Y_in_Y} shows that, if $y\in {\mathscr Y}_{\alpha,\eta}(0,T)$ and  $z=S(\cdot)\psi+\mathscr I_{SGy}(0,\cdot)$ in $[0,T]$, then $z$ admits SG-derivative and $\check z=Gy$.}
\end{remark}

Now we are able to state the main result of this section. 

\begin{theorem}
\label{thm:ex_un_mild_solution_rough}
For every $\psi\in E_{\eta+\alpha}$ the Cauchy problem \eqref{cauchy_prob_rough} admits a unique mild solution $y\in {\mathscr Y}_{\alpha,\eta}(0,T)$. Further, 
$y(t)$ belongs to $E_{1+\mu}$ 
for every $\mu\in[0,2\eta+\alpha-1)$ and $t\in(0,T]$, and there exists a positive constant $\mathfrak C$, which depends on $\mu$, $T$, $\alpha$, $\eta$, $G$, $[x]_{C^{\eta}([0,T])}$, $\|\mathbb{X} \|_{\mathscr C^{2\eta}([0,T]^2_{<})}$ and
the constants appearing in 
Hypotheses $\ref{hyp-main}$, such that
\begin{align}
\|y(t)\|_{E_{1+\mu}}\leq \mathfrak Ct^{\eta+\alpha-1-\mu}, \qquad\;\, t\in(0,T].
\label{estim-1}
\end{align}
\end{theorem}
\begin{proof}
We split the proof into two steps. In the former, we prove existence and uniqueness of the mild solution to the Cauchy problem \eqref{cauchy_prob_rough}, while, in the latter, we study the smoothness of the solution.

{\em Step 1}. We begin by showing that the operator $\Gamma$, defined in \eqref{def_gamma_op}, is a $\frac12$-contraction on the closed subset $\mathcal B$ of ${\mathscr Y}_{\alpha,\eta}(0,\overline T)$ given by
\begin{align*}
\mathcal B=\big\{y\in {\mathscr Y}_{\alpha,\eta}(0,\overline T):y(0)=\psi,\ \|y\|_{{\mathscr Y}_{\alpha,\eta}(0,\overline T)}\le {\mathfrak R}|\psi|_{E_{\eta+\alpha}}\big\},
\end{align*}
where $\overline{T}$ is defined by \eqref{choice_ovT}, i.e.,
\begin{eqnarray*}
\overline{T}=(2\mathfrak{C}_4(1+[x]^2_{C^{\eta}([0,T])}+\|\mathbb{X}\|_{\mathscr C^{2\eta}([0,T]^2_{<})}))^{-\frac{1}{\eta-\alpha}}
\end{eqnarray*}
$\mathfrak{C}_4$ and $\mathfrak R$ are the  constant which appear in \eqref{stima_parz_mild_sol} and \eqref{stima_a_priori_Y}, respectively.
We stress that $\mathcal B$ is not empty, since it contains $\Gamma(0)$. 

Let us prove that $\Gamma$ is a $\frac{1}{2}$-contraction in $\mathcal{B}$.
First of all, we observe that $(\Gamma(y))(0)=\psi(0)$ for every $y\in\mathcal{B}$ and Proposition \ref{prop:Gamma_mappa_Y_in_Y} guarantees that $\Gamma$ maps $\mathcal{B}$ into itself. It remains to prove that $\|\Gamma(y)\|_{{\mathscr Y}_{\alpha,\eta}(0,\overline T)}\leq {\mathfrak R}|\psi|_{E_{\eta+\alpha}}$.

From \eqref{stima_parz_mild_sol}, with $[a,b]$ replaced by $[0,\overline T]$ and the definition of $\mathfrak{R}$ (see the end of the proof of Proposition \ref{prop:Gamma_mappa_Y_in_Y}), which implies that
\begin{align*}
(1+K_{\eta+\alpha,\alpha}\|G\|_{\mathscr{L}(E_{\eta+\alpha})} )L_{\eta+\alpha,\eta+\alpha,T}\leq \frac12\mathfrak R,    
\end{align*}
we infer that
the choice of $\overline T$ gives
\begin{align}
\|\Gamma(y)\|_{{\mathscr Y}_{\alpha,\eta}(0,\overline T)}
\leq &(1+K_{\eta+\alpha,\alpha}\|G\|_{\mathscr{L}(E_{\eta+\alpha})} )L_{\eta+\alpha,\eta+\alpha,T}|y(0)|_{E_{\eta+\alpha}}+\frac{1}{2}\|y\|_{{\mathscr Y}_{\alpha,\eta}(0,\overline T)}\notag \\
\leq & \frac{1}{2}\mathfrak{R}|\psi|_{E_{\eta+\alpha}}+\frac{1}{2}{\mathfrak R}|\psi|_{E_{\eta+\alpha}}= {\mathfrak R}|\psi|_{E_{\eta+\alpha}}.
\label{corradini}
\end{align}

This means that $\mathcal B$ is invariant under the action of $\Gamma$.

Showing that $\Gamma:\mathcal B\to \mathcal B$ is a $\frac12$-contraction, is a straightforward consequence of \eqref{corradini}.
Indeed, by Lemma \ref{lemma:lin_int}(i), it follows that $\mathscr{I}_{SGy_2}-\mathscr{I}_{SGy_1}=\mathscr{I}_{SG(y_2-y_1)}$, so that $\Gamma(y_2)-\Gamma(y_1)=\Gamma(y_2-y_1)$ and the function $y_2-y_1$ belongs to $\mathscr{ Y}_{\alpha,\eta}(0,\overline{T})$ since both $y_1$ and $y_2$ belong to such a space.
Moreover, the functions $y_2-y_1$ and $\Gamma(y_2)-\Gamma(y_1)$ vanish at $t=0$, due to the fact that $y_1$, $y_2$, $\Gamma(y_1)$ and $\Gamma(y_2)$ belong to $\mathcal B$. Hence,
estimate \eqref{corradini}, with $y=y_2-y_1$, shows that $\|\Gamma(y_2)-\Gamma(y_1)\|_{\mathscr{Y}_{\alpha,\eta}(0,\overline{T})}\le\frac{1}{2}\|y_2-y_1\|_{\mathscr{Y}_{\alpha,\eta}(0,\overline{T})}$, i.e. $\Gamma$ is a $\frac{1}{2}$-contraction in $\mathcal{B}$ 
as it has been claimed.

Since $\Gamma$ is a $\frac{1}{2}$-contraction on $\mathcal B$, it admits a unique fixed point $y_0$ and so we get the unique mild solution to \eqref{cauchy_prob_rough} in ${\mathscr Y}_{\alpha,\eta}(0,\overline T)$. 

If $T\le\overline T$, then we are done, otherwise we 
argue by recurrence. For this purpose,
we introduce the operator $\Gamma_1:\mathcal{B}_1\to \mathscr Y_{\alpha,\eta}(\overline T,T_1)$, defined by
\begin{align*}
(\Gamma_1(y))(t)=S(t-\overline T)y_0(\overline T)+\mathscr I_{SGy}(\overline{T},t),\qquad\;\, t\in[\overline T, T_1],    
\end{align*}
for every $y$ in the subset $\mathcal{B}_1\subset {\mathscr Y}_{\alpha,\eta}(\overline T,T_1)$, whose elements satisfy the conditions $y(\overline T)=y_0(\overline T)$ and $\|y\|_{{\mathscr Y}_{\alpha,\eta}(\overline T,T_1)}\leq {\mathfrak R}|y_0(\overline T)|_{E_{\eta+\alpha}}$, 
where $T_1=\min\{2\overline T,T\}$. $\mathcal B_1$ is not empty, since the function 
$y=S(\cdot-\overline T)y_0(\overline{T})$
belongs to $\mathcal{B}_1$. This follows arguing as above.

Since the constant $\overline T$ depends only on $\alpha$, $\eta$, $T$, $[x]_{C^{\eta}([0,T])}$, $\|\mathbb{X}\|_{\mathscr C^{2\eta}([0,T]^2_{<})}$, arguing as in the first part of the proof, we can easily show that $\Gamma_1$ is a $\frac{1}{2}$-contraction in $\mathcal{B}_1$ and admits a unique fixed point $y_1$. 

Let us introduce the function $y:[0,T_1]\to E$ defined by
\begin{align*}
y(t)=\begin{cases}
y_0(t), & t\in[0,\overline T], \\
y_1(t), & t\in(\overline T,T_1].
\end{cases}
\end{align*}
Clearly, $y\in C((0,T_1];E_{\eta+\alpha})\cap B([0,T_1];E_{\eta+\alpha})$, Moreover, 
$Gy$ is its SG-derivative. Indeed, there exist functions $R_0\in \mathscr{C}^{\eta+\alpha}([0,\overline{T}]^2_{<};E_{\alpha})$ and $R_1\in \mathscr{C}^{\eta+\alpha}([\overline{T},T_1]^2_{<};E_{\alpha})$ such that
\begin{align*}
&(\hat\delta_1y)(s,t)=S(t-s)Gy(s)(x(t)-x(s))+R_0(s,t),\qquad\;\,(s,t)\in [0,\overline T]^2_{<},\\
&(\hat\delta_1y)(s,t)=S(t-s)Gy(s)(x(t)-x(s))+R_1(s,t),\qquad\;\,(s,t)\in [\overline T,T_1]^2_{<}.
\end{align*}

Suppose that $s\in [0,\overline{T})$ and $t\in (\overline{T},T_1]$. Then, 
\begin{align}
(\hat\delta_1y)(s,t)=&y(t)-S(t-s)y(s)\notag\\
=&(\hat\delta_1y)(\overline{T},t)+S(t-\overline T)(\hat\delta_1y)(s,\overline T)\notag\\
=&S(t-\overline{T})Gy(\overline{T})(x(t)-x(\overline T))+R_1(\overline T,t)\notag\\
&+S(t-\overline{T})(S(\overline T,s)Gy(s)(x(\overline T)-x(s))+R_0(s,\overline T))\notag\\
=&S(t-s)Gy(s)(x(t)-x(s))+
S(t-\overline T)(\hat\delta_1Gy)(s,\overline T)(x(t)-x(\overline{T}))\notag\\
&+R_1(\overline T,t)+S(t-\overline T)R_0(s,\overline T).
\label{holyanna}
\end{align}
Since by assumptions $\hat\delta_1\check y\in \mathscr{C}^{\alpha}([0,\overline{T}]^2_{<};E_{\alpha})$, taking Hypotheses \ref{hyp-main} into account, it is immediate to check that
\begin{align*}
&|S(t-\overline T)(\hat\delta_1Gy)(s,\overline T)(x(t)-x(\overline{T}))+R_1(\overline T,t)+S(t-\overline T)R_0(s,\overline T)|_{E_{\alpha}}\\
\le & L_{\alpha,\alpha}\|\hat\delta_1Gy\|_{\mathscr C^{\alpha}([0,\overline T]^2_{<};E_{\alpha})}[x]_{C^{\eta}([0,T])}|\overline T-s|^{\alpha}|t-\overline{T}|^{\eta}
+\|R_1\|_{\mathscr{C}^{\eta+\alpha}([\overline T,T_1]^2_{<})}|t-\overline T|^{\eta+\alpha}\\
&+L_{\alpha,\alpha}\|R_0\|_{\mathscr{C}^{\eta+\alpha}([0,\overline T]^2_{<};E_{\alpha})}|\overline T-s|^{\eta+\alpha}\\
\le & [L_{\alpha,\alpha}(\|\hat\delta_1Gy\|_{\mathscr C^{\alpha}([0,\overline T]^2_{<};E_{\alpha})}[x]_{C^{\eta}([0,T])}
+\|R_0\|_{\mathscr{C}^{\eta+\alpha}([0,\overline T]^2_{<};E_{\alpha})})\\
&\;\;+\|R_1\|_{\mathscr{C}^{\eta+\alpha}([\overline T,T_1]^2_{<};E_{\alpha})}]|t-s|^{\eta+\alpha},
\end{align*}
where $L_{\alpha,\alpha}=L_{\alpha,\alpha,T}$.
We thus conclude that the function $(s,t)\mapsto(\hat\delta_1y)(s,t)-S(t-s)Gy(s)(x(t)-x(s))$ belongs to $\mathscr{C}^{\eta+\alpha}([0,T_1]^2_{<};E_{\alpha})$. Therefore, the function $Gy$ is the SG-derivative of $y$. In particular, $\check y\in C((0,T_1];E_{\eta+\alpha})\cap B([0,T_1];E_{\eta+\alpha})$.

Further, we observe that $\hat\delta_1\check y=\hat\delta_1Gy$ belongs to $\mathscr{C}^{\alpha}([0,T_1];E_{\alpha})$. Indeed, the computations in the first part of \eqref{holyanna} can be repeated to show that
\begin{align*}
(\hat\delta_1\check y)(s,t)=
(\hat\delta_1\check y)(\overline{T},t)+S(t-\overline T)(\hat\delta_1\check y)(s,\overline T)
\end{align*}
for every $(s,t)\in [0,T_1]^2_{<}$, with $s<\overline{T}<t$. From this formula we deduce that
\begin{align*}
|(\hat\delta_1\check y)(s,t)|_{E_{\alpha}}\le &
|(\hat\delta_1\check y)(\overline{T},t)|_{E_{\alpha}}+L_{\alpha,\alpha}|(\hat\delta_1\check y)(s,\overline T)|_{E_{\alpha}}\\
\le &(\|\hat\delta_1\check y\|_{\mathscr{C}^{\alpha}([\overline{T},T_1]^2_{<};E_{\alpha})}+L_{\alpha,\alpha}
\|\hat\delta_1\check y\|_{\mathscr{C}^{\alpha}([0,\overline{T}]^2_{<};E_{\alpha})})|t-s|^{\alpha}.
\end{align*}

This formula, combined with the fact that $\hat\delta_1\check y$ belongs to $\mathscr{C}^{\alpha}([0,\overline{T}]^2_{<};E_{\alpha})$ and
to $\mathscr{C}^{\alpha}([\overline{T},T_1]^2_{<};E_{\alpha})$, allows us to conclude that $y$ belongs to $\mathscr{Y}_{\alpha,\eta}(0,T_1)$. Finally, we claim that 
\begin{align}
\label{sol_mild_y_unione}
y(t)=S(t)\psi+\mathscr I_{SGy}(0,t)     
\end{align}
for every $t\in [0,T_1]$. Clearly, this formula holds true if $t\in [0,\overline{T}]$. On the other hand, if $t\in [\overline{T},T_1]$, then 
\begin{align*}
y(t)
= & S(t-\overline T)y_0(\overline T)+\mathscr I_{SGy_1}(\overline T,t) \notag \\
= & S(t-\overline T)(S(\overline T)\psi+\mathscr I_{SGy}(0,\overline T))+\mathscr I_{SGy}(\overline T,t).
\end{align*}
Using the semigroup property and
\eqref{spez_int}, we conclude that \eqref{sol_mild_y_unione} holds true also in the interval 
$[\overline{T},T_1]$. 

We have so proved that function $y$ is a mild solution to the Cauchy problem \eqref{cauchy_prob_rough}. If $T_1=T$, then we are done, otherwise we repeat this procedure. 
Let $\overline n\in\N$ be the minimum integer such that $\overline n\overline T\geq T$. We iterate the above arguments to construct functions $y_j$, defined in $[j\overline T,(j+1)\overline T]$, with $j=0,\ldots,\overline n-1$ and then define the function $y:[0,T]\to E$ by setting 
\begin{align*}
y(t)=
\begin{cases}
y_0(t), & t\in[0,\overline T], \\
y_1(t), & t\in (\overline T,2\overline T], \\
\ \ \vdots & \ \ \ \ \ \vdots \\
y_{n-1}(t), & t\in ((\overline n-1)\overline T,T].
\end{cases}
\end{align*}

The function $y$ turns out to be the mild solution to \eqref{cauchy_prob_rough} (the uniqueness follows from estimate \eqref{stima_a_priori_Y}).

{\em Step 2.} Here, we study the regularity properties of the mild solution $y$ to \eqref{cauchy_prob_rough}, obtained in Step 1. To prove estimate \eqref{estim-1}, we use a bootstrap argument.

At first we notice that, from \eqref{stima_int_completa-1} and Hypotheses \ref{hyp-main}(iii)-(a), we can infer that
$y(t)\in E_{\xi}$ for every $t\in (0,T]$
and $\xi\in [\eta+\alpha,1)$, and
\begin{align}
|y(t)|_{E_\xi}
= & |S(t)\psi+\mathscr I_{SGy}(0,t)|_{E_\xi} \notag \\
\leq &  L_{\eta+\alpha,\xi,T}|\psi|_{E_{\eta+\alpha}}t^{\eta+\alpha-\xi}+\|\mathscr I_{SGy}\|_{\mathscr C^{2\eta+\alpha-\xi}([0,T]^2_<;E_\xi)}T^{2\eta+\alpha-\xi} \notag \\
\leq & \mathfrak C_1 t^{\eta+\alpha-\xi}
\label{stima_reg_y_1}
\end{align}
for every $t\in(0,T]$ and $\xi\in [\eta+\alpha,1)$, where $\mathfrak C_1$ is a positive constant which depends on $T$, $\alpha$, $\eta$, $G$, $[x]_{C^{\eta}([0,T])}$, $\|\mathbb X\|_{\mathscr{C}^{2\eta}([0,T]^2_{<})}$ and
the constants which appear in Hypotheses \ref{hyp-main}. 

To prove that actually, $y(t)\in E_{1+\mu}$ for every $t\in (0,T]$ and $\mu\in [0,2\eta+\alpha-1)$, we use \eqref{stima_reg_y_1}, together with Lemma \ref{lem:reg_int_conv_rough} with 
$\check f=G\check y$ and $R^{F}=R^{GY}$ and a proper choice of $\lambda>0$,  $\omega=\alpha$ and $\omega_1=\eta+\alpha$. More precisely, we take $\lambda$ in the open interval $((\eta+2\alpha-1)\vee 0, 2\eta+\alpha-1)$. Note that this interval is not empty, since
$\alpha<\eta$ and $2\eta+\alpha>1$, and $\lambda<\alpha$ since $2\eta<1$.
We aim at proving that there exist positive constants $\mathfrak{C}_2$ and $\mathfrak{C}_3$,
which depend at most on $T$, $\alpha$, $\eta$, $G$, $[x]_{C^{\eta}([0,T])}$, $\|\mathbb{X}\|_{\mathscr C^{2\eta}([0,T]^2_{<})}$ and the constants in Hypotheses \ref{hyp-main}, such that
\begin{align}
&\|\hat\delta_1G\check{y}\|_{\mathscr{C}^{\alpha}([0,T]^2_{<};E_{\alpha})}\le\mathfrak{C}_2,
\label{giulio}
\\[1mm]
&\|R^{GY}\|_{\mathscr C^{\eta+\alpha-\lambda}([\varepsilon,T]^2_<;E_{\lambda})}\leq \mathfrak C_3 \varepsilon^{\lambda-\alpha}, \qquad \;\, \varepsilon\in(0,T).
\label{cesare}
\end{align}

Once these estimates are proved, we can apply Lemma \ref{lem:reg_int_conv_rough} and conclude that $\mathscr I_{SGy}\in \mathscr C^{2\eta+\alpha-\xi}([\varepsilon,T]^2_<;E_{\xi})$ for every $\xi\in[1,\lambda+1)$ and there exists a positive constant $\mathfrak C_4$, which depends on $T$, $\alpha$, $\eta$, $G$, $[x]_{C^{\eta}([0,T])}$, $\|\mathbb{X}\|_{\mathscr C^{2\eta}([0,T]^2_{<})}$ and the constants in Hypotheses \ref{hyp-main}, but not on $\varepsilon$, such that
\begin{align}
\label{stima_reg_y_int_epsilon}
 \|\mathscr I_{SGy}\|_{\mathscr C^{2\eta+\alpha-\xi}([\varepsilon,T]^2_<;E_\xi)}
 \leq \mathfrak C_4 \varepsilon^{\lambda-\alpha}.
\end{align}

We begin by estimating $\hat\delta_1 G\check y$. Since $y=\Gamma(y)$, it follows that $\check y=Gy\in B([0,T];E_{\eta+\alpha})$. Therefore, arguing as in \eqref{stima_delta_1Gchecky} we get
\begin{align*}
|(\hat\delta_1G\check{y})(s,t)|_{E_\alpha}   \leq & \mathfrak{g}_{\alpha}\|\hat\delta_1\check{y}\|_{\mathscr C^\alpha([0,T]^2_<;E_\alpha)}|t-s|^{\alpha}  \\
& +C_{\eta+\alpha,\alpha,T}\mathfrak{g}_{\alpha,\eta+\alpha}\|\check{y}\|_{B([0,T];E_{\eta+\alpha})}|t-s|^\eta 
\end{align*}
for every $(s,t)\in [0,T]^2_<$. Hence, $\hat\delta_1G\check y$ belongs to $\mathscr C^{\alpha}([0,T]^2_<;E_{\alpha})$ and estimate \eqref{giulio} follows.

As far as $R^{GY}$ is concerned, we adapt the arguments used in the proof of  \eqref{stima_S_der_Gy} with a different choice of the parameters. By \eqref{forma_rem_GY}, $R^{GY}$ is split into the sum of three terms which we estimate separately.
Recalling that $\check y=Gy$ and $E_{\alpha}$ is continuously embedded into $E_{\lambda}$, since $\lambda<\alpha$, we get
\begin{align}
&|[G,\mathfrak a(s,t)]\check y(s)(x(t)-x(s))|_{E_{\lambda}}\notag\\
\le & K_{\alpha,\lambda}|[G,\mathfrak a(s,t)]\check y(s)(x(t)-x(s))|_{E_\alpha} \notag \\
\leq & 
K_{\alpha,\lambda}
\mathfrak{g}_{\alpha,\eta+\alpha}C_{\eta+\alpha,\alpha,T}[x]_{C^\eta([0,T])}\|\check{y}\|_{B([0,T];E_{\eta+\alpha})}|t-s|^{2\eta} 
\label{stima_reg_M_1}
\end{align}
for every $(s,t)\in[0,T]^2_<$. Similarly, taking into account \eqref{stima_reg_y_1}, with $\xi=\eta+2\alpha-\lambda$ (which belongs to the interval $[\eta+\alpha,1)$, due to the condition imposed on $\lambda$) and recalling that the fixed point of operator $\Gamma$ satisfies \eqref{stima_a_priori_Y}, we can estimate
\begin{align}
|[G,\mathfrak a(s,t)]y(s)|_{E_\lambda}
\leq & K_{\alpha,\lambda}\mathfrak{g}_{\alpha}C_{\eta+2\alpha-\lambda,\alpha,T}|y(s)|_{E_{\eta+2\alpha-\lambda}}|t-s|^{\eta+\alpha-\lambda} \notag \\
& +C_{\eta+\alpha,\lambda,T}\mathfrak{g}_{\eta+\alpha}|y(s)|_{E_{\eta+\alpha}}|t-s|^{\eta+\alpha-\lambda}\notag\\
\le &
K_{\alpha,\lambda}\mathfrak{g}_{\alpha}C_{\eta+2\alpha-\lambda,\alpha,T}\mathfrak{C}_1s^{\lambda-\alpha}|t-s|^{\eta+\alpha-\lambda} \notag \\
& +C_{\eta+\alpha,\lambda,T}\mathfrak{g}_{\eta+\alpha}\mathfrak{R}|\psi|_{E_{\eta+\alpha}}|t-s|^{\eta+\alpha-\lambda}
\label{stima_reg_M_3}
\end{align}
for every $(s,t)\in(0,T]^2_<$. 
Finally, the term $R^{GY}$ can be estimated as follows:
\begin{align}
\label{stima_reg_M_2}
|GR^Y(s,t)|_{E_\alpha}
\leq \mathfrak{g}_{\alpha}\|R^Y\|_{\mathscr C^{\eta+\alpha}([0,T]^2_<;E_\alpha)}|t-s|^{\eta+\alpha}, \qquad\;\, (s,t)\in[0,T]^2_<.
\end{align}

From \eqref{stima_reg_M_1}-\eqref{stima_reg_M_2}, estimate \eqref{cesare} follows easily, recalling that $\alpha<\eta$. 

Using \eqref{stima_reg_y_int_epsilon} we can complete the proof of \eqref{estim-1}. For this purpose, we fix $\mu\in[0,2\eta+\alpha-1)$, $\varepsilon\in (0,T/2)$ and split 
\begin{align*}
y(t)=S(t-\varepsilon)y(\varepsilon)+\mathscr I_{SGy}(\varepsilon,t), \qquad\;\, t\in [\varepsilon,T]. 
\end{align*}

Since \eqref{stima_reg_y_int_epsilon} holds true for every $\xi\in [1,\lambda+1)$, if we take $\lambda=\frac{1}{2}((2\eta+\alpha-1)+(\mu\vee(\eta+2\alpha-1)))$, then
$\lambda\in ((\eta+2\alpha-1)\vee 0,2\eta+\alpha-1)$ and
$1+\mu$ belongs to the interval $[1,\lambda+1)$
Hence, from Hypothesis \ref{hyp-main}(iii)-(a), \eqref{stima_a_priori_Y} and \eqref{stima_reg_y_int_epsilon}, with $\xi=1+\mu$, we infer that $y(t)\in E_{1+\mu}$ for every $t\in[\varepsilon,T]$ and 
\begin{align*}
|y(t)|_{E_{1+\mu}}
\leq & L_{\eta+\alpha,1+\mu,T}(t-\varepsilon)^{\eta+\alpha-1-\mu}|y(\varepsilon)|_{E_{\eta+\alpha}}
+|\mathscr I_{SGy}(\varepsilon,t)|_{E_{1+\mu}} \\
\leq & L_{\eta+\alpha,1+\mu,T}\mathfrak R |\psi|_{E_{\eta+\alpha}}\varepsilon^{\eta+\alpha-1-\mu}+\mathfrak C_2T^{2\eta+\alpha-1-\eta}\varepsilon^{\lambda-\alpha}
\end{align*}
for every $t\in[2\varepsilon,T]$.

The conditions on $\lambda$ imply that
$\eta+\alpha-1-\mu<\lambda-\alpha$. Hence, from the previous estimate it follows that there exists a positive constant $\mathfrak C_5$, which depends on $T$, $\alpha$, $\eta$, $\mu$, $G$, $[x]_{C^{\eta}([0,T])}$, $\|\mathbb X\|_{\mathscr C^{2\eta}([0,T]^2_{<})}$ and the constants in Hypotheses \ref{hyp-main}, but it is independent of  $\varepsilon$, such that
\begin{align*}
|y(t)|_{E_{1+\mu}}\leq \mathfrak C_5\varepsilon^{\eta+\alpha-1-\mu}, \qquad\;\, t\in[2\varepsilon,T].    
\end{align*}
Choosing $t=2\varepsilon$, estimate  \eqref{estim-1} follows immediately.
\end{proof}

\begin{remark}
{\rm 
The smoothing effect of the equation on the initial datum has been claimed also in \cite{GHN} in a more general framework, albeit with no proof.}    
\end{remark}

\begin{remark}
\label{rmk:cont_y_D(A)}
{\rm The unique mild solution $y$ to \eqref{cauchy_prob_rough} is continuous in $(0,T]$ with values in $E_1=D(A)$. Indeed, from \eqref{stima_reg_y_int_epsilon} we know that $\mathscr I_{SGy}\in\mathscr C^{2\eta+\alpha-1}([\varepsilon,T]^2_<;E_1)$, for every $\varepsilon\in(0,T)$, and Remark \ref{strong-cont-smgr} shows that the function $S(\cdot)\psi$ is continuous in $(0,\infty)$ with values in $E_1$.}
\end{remark}

The following remark will be used in the next section to approximate the solution $y$ provided by Theorem \ref{cauchy_prob_rough} with smoother solutions.

\begin{remark}
\label{rmk:es_sol_mild_2}
{\rm We stress again that, without condition \eqref{cond-x}, the SG-derivative of a function $y$, when existing, is not unique. Nevertheless, we can still prove the existence of a unique solution $y$ to the Cauchy problem \eqref{cauchy_prob_rough} with SG-derivative $Gy$. The main difference is in the definition of both the space $\mathscr{Y}_{\alpha,\eta}(0,T)$ and the operator $\Gamma$. Indeed, in this situation, ${\mathscr Y}_{\alpha,\eta}(0,T)$ is the set of all pairs $Y=(y,z)$ such that
$y\in C((0,T];E_{\eta+\alpha})\cap B([0,T];E_{\eta+\alpha})$, $z\in C((0,T];E_\alpha) \cap  B([0,T];E_\alpha)$, $z$ is a SG-derivative of $y$ with remainder  $R^Y\in \mathscr C^{\eta+\alpha}([0,T]^2_<;E_\alpha)$ and $\hat\delta_1z$ belongs to $\mathscr C^\alpha([0,T]^2_<;E_\alpha)$.
${\mathscr Y}_{\alpha,\eta}(0,T)$ is a Banach space when endowed with the norm
\begin{align*}
\|Y\|_{\mathscr Y_{\alpha,\eta}(0,T)}
= & \|y\|_{B([0,T];E_{\eta+\alpha})}+ \|z\|_{B([0,T];E_\alpha)}+\|R^Y\|_{\mathscr C^{\eta+\alpha}([0,T]^2_<;E_\alpha)} \\
& +\|\hat\delta_{1}z\|_{\mathscr C^\alpha([0,T];E_\alpha)}
\end{align*}
for every 
$Y=(y,z)\in {\mathscr Y}_{\alpha,\eta}(0,T)$. 

The operator $\Gamma$ is defined by 
$\Gamma(Y)=(S(\cdot)\psi+\mathscr I_{SGy},Gy)$ for every $Y=(y,z)\in {\mathscr Y}_{\alpha,\eta}(0,T)$ (this operator is well-defined thanks to Remark \ref{rmk:un_rough_young_int}(i)). Adapting the arguments in the proofs of Proposition \ref{prop:Gamma_mappa_Y_in_Y} and Theorem \ref{thm:ex_un_mild_solution_rough}, it can be easily proved that
problem \eqref{cauchy_prob_rough} admits a solution $y\in C((0,T];E_{\eta+\alpha})\cap B([0,T];E_{\eta+\alpha})$ with the function $Gy$ as SG-derivative.
To achieve this result is suffices to show that $\Gamma$ is a contraction mapping in the space
\begin{align*}
\mathcal B=\big\{Y\in {\mathscr Y}_{\alpha,\eta}(0,\overline T):y(0)=\psi, \ z(0)=G\psi,\ \|Y\|_{{\mathscr Y}_{\alpha,\eta}(0,\overline T)}\le {\mathfrak R}|\psi|_{E_{\eta+\alpha}}\big\},
\end{align*} 
which is not empty since it contains the pair
$(S(\cdot)\psi+\mathscr{I}_{SGS\psi}(0,\cdot),GS(\cdot)\psi)$
(see Remark \ref{rem-utile}).}
\end{remark}

\section{Integral representation of the mild solution}
\label{sect-4}
In this section we assume Hypotheses 
\ref{hyp-main} and \ref{hyp-3.1}, and show the continuity of the mild solution $y$ to \eqref{cauchy_prob_rough} with respect to the pair $(x,\mathbb X)$, under a suitable topology. We also prove that, if $(x,\mathbb X)$ is a geometric rough path, then the mild solution $y$ to \eqref{cauchy_prob_rough} admits an integral formulation, i.e., $y$ satisfies the equation
\begin{align}
y(t)=\psi+\int_0^tAy(s)ds+\mathscr I_{Gy}(0,t) \qquad\;\, t\in[0,T],    
\label{rappr_int_sol}
\end{align}
where $\mathscr I_{Gy}$ is the rough integral of $y$ with respect to $(x,\mathbb X)$ defined in Proposition \ref{prop:old_sew_map}. 

We find it useful to introduce the following definition.

\begin{definition}
We denote by $\mathcal D^{\eta}([a,b])$ the space of couples $X=(x,\mathbb X)$ such that $x\in C^\eta([a,b])$ and $\mathbb X\in \mathscr C^{2\eta}([a,b]^2_{<})$ satisfies condition \eqref{form-diff-X}, i.e.,
\begin{eqnarray*}
\mathbb{X}(s,u)-\mathbb{X}(s,t)-\mathbb{X}(t,u)=(x(t)-x(s))(x(u)-x(t)),
\end{eqnarray*}
for every $(s,t,u)\in [a,b]^3_{<}$. We set
\begin{align*}
[x]_{\mathcal D^{\eta}([a,b])}:=[x]_{C^\eta([a,b])}+\|\mathbb X\|_{\mathscr C^{2\eta}([a,b]^2_<)}    
\end{align*}
for every $X=(x,\mathbb X)\in \mathcal D^{\eta}([a,b])$.
\end{definition}


\begin{definition}
A couple $(x,\mathbb X)\in \mathcal D^{\eta}([a,b])$ is a geometric rough path if 
\begin{align*}
\mathbb X(s,t)=\frac12(x(t)-x(s))^2, \qquad (s,t)\in[0,T]^2_<.        
\end{align*}
\end{definition}

The following result will be very useful for our purposes.

\begin{proposition}
\label{conv_rough_paths}
For every $X=(x,\mathbb X)\in \mathcal D^{\eta}([a,b])$, there exist a sequence $(x_n,\mathbb X_n)_{n\in\N}\subset \mathcal D^{\eta}([a,b])$ of geometric rough paths, with $x_n\in C^1([a,b])$, and a function $h\in C^{2\eta}([a,b])$ such that the sequence $(x_n,\mathbb X_n+\delta_1h)_{n\in\N}$ is bounded in $C^{\eta}([a,b])\times\mathscr{C}^{2\eta}([a,b]^2_{<})$ and converges to $X=(x,\mathbb X)$ in $C^{\eta'}([a,b])\times\mathscr{C}^{2\eta'}([a,b]^2_{<})$ for every $\eta'<\eta$. Further, we may assume $\delta_1h=0$ if $X$ is a geometric rough path.
\end{proposition}

\begin{proof}
The proof is straightforward. Indeed, it suffices to extend $x$ with a function $x\in C^{\eta}_b(\R)$ (this can be done, for instance, by setting $x(t)=x(a)$ for every $t<a$ and $x(t)=x(b)$ for every $t>b$). We still denote by $x$ the so obtained function and we regularize it by convolution with a standard sequence of mollifiers. It is well known that the so obtained sequence $(x_n)_{n\in\N}$ converges uniformly in $\R$ to $x$ and $\|x_n\|_{C^{\eta}([a,b])}\le
\|x\|_{C^{\eta}([a,b])}$ for every $n\in\N$. Moreover, since
\begin{eqnarray*}
\|x_n-x\|_{C^{\eta'}([a,b])}\le C_{\eta',\eta}\|x_n-x\|_{\infty}^{1-\frac{\eta'}{\eta}}\|x_n-x\|_{C^{\eta}([a,b])}^{\frac{\eta'}{\eta}}  
\end{eqnarray*}
for every $n\in\N$, it follows easily that $x_n$ converges to $x$ in $C^{\eta'}([a,b])$ for every $\eta'<\eta$.

Now, it is easy to show that the sequence $(\widetilde{\mathbb X}_n)$ defined by $\widetilde{\mathbb X}(s,t)=\frac{1}{2}(x(t)-x(s))^2$ for every $(s,t)\in [a,b]^2$ is bounded in $\mathscr{C}^{2\eta}([a,b]^2)$ and it converges in $\mathscr{C}^{2\eta'}([a,b]^2_<)$ to the function 
$\widetilde{\mathbb X}$ defined by $\widetilde{\mathbb X}(s,t)$ for every $(s,t)\in [a,b]^2_{<}$. The assertion follows from observing that $\mathbb{X}=\widetilde{\mathbb X}+\delta_1h$ for some function $h\in C^{2\eta}([a,b])$ (see Remark \ref{rem-country}.
\end{proof}

Next lemma shows that $\mathscr I_{Gy}$ is well-defined if $y\in {\mathscr Y}_{\alpha,\eta}(0,T)$ with $\alpha>1-2\eta$.

\begin{lemma}
\label{lem:esistenza_int_rough}
Fix $\alpha>1-2\eta$ and $y\in {\mathscr Y}_{\alpha,\eta}(0,T)$. Then, the following properties are satisfied.
\begin{enumerate}[\rm (i)]
\item 
$Gy$ admits Gubinelli derivative $G\check y\in C^\alpha([0,T];E)$ with remainder $\widetilde R^{GY}$, which belongs to $\mathscr C^{\eta+\alpha}([0,T]^2_<;E_{\alpha})$ and is given by
\begin{align}
\label{gub_deriv_Y}
\widetilde R^{GY}(s,t)
= & G\mathfrak a(s,t)\check y(s)(x(t)-x(s))+G\mathfrak a(s,t) y(s)+ GR^{Y}(s,t) \end{align}
for every $(s,t)\in [0,T]^2_{<}$;
\item the rough integral $\mathscr I_{Gy}$ is well-defined and belongs to $\mathscr C^{2\eta+\alpha}([0,T]^2_<;E)$.    
\end{enumerate}
\end{lemma}

\begin{proof}
(i) We fix $y\in\mathscr{Y}_{\alpha,\eta}(0,T)$ and observe that, for every $(s,t)\in[0,T]^2_<$, it holds that
\begin{align*}
Gy(t)-Gy(s)
= & G(\hat\delta_1y)(s,t)+G\mathfrak a(s,t)y(s) \\
= & GS(t-s)\check y(s)(x(t)-x(s))+G\mathfrak a(s,t)y(s)+ GR^Y(s,t) \\
= & G\check y(s)(x(t)-x(s))+G\mathfrak a(s,t)\check y(s)(x(t)-x(s))+G\mathfrak a(s,t) y(s) \\
& + GR^{Y}(s,t)\\
=& G\check y(s)(x(t)-x(s))+\widetilde R^{GY}(s,t).
\end{align*}
Note that 
\begin{align}
|\widetilde R^{GY}(s,t)|_E    
\leq \mathfrak{g}_0\big(&C_{\alpha,0,T}[x]_{C^\eta([0,T])}\|\check y\|_{B([0,T];E_\alpha)}+C_{\eta+\alpha,0,T}\|y\|_{B([0,T];E_{\eta+\alpha})}\notag \\
& +K_{\alpha,0}\|R^{Y}\|_{\mathscr C^{\eta+\alpha}([0,T]^2_<;E_\alpha)}\big)|t-s|^{\eta+\alpha}
\label{intro-mario}
\end{align}
for every $(s,t)\in[0,T]^2_<$. This implies that $Gy$ admits Gubinelli derivative $G\check y$ with respect to $x$ with remainder $\widetilde R^{GY}$ and 
\begin{align*}
\|\widetilde R^{GY}\|_{\mathscr C^{\eta+\alpha}([0,T]^2_<;E)}\leq (C_{\alpha,0,T}\vee C_{\eta+\alpha,0,T}\vee K_{\alpha,0})(1+[(x,\mathbb X)]_{\mathcal D^{\eta}([0,T])})\|y\|_{{\mathscr Y}_{\alpha,\eta}(0,T)}.
\end{align*}
Further, for every $(s,t)\in[0,T]^2_<$ we can estimate
\begin{align*}
|G\check y(t)-G\check y(s)|_{E}
=& |G(\hat\delta_1\check y)(s,t)+G\mathfrak a(s,t)\check y(s)|_E \\
\leq & \mathfrak{g}_0(K_{\alpha,0}\|\hat\delta_1\check y\|_{\mathscr C^\alpha([0,T];E_\alpha)}+C_{\alpha,0,T}\|\check y\|_{B([0,T];E_\alpha)})|t-s|^\alpha,
\end{align*}
so that $G\check{y}$ belongs to $\mathscr{C}^{\alpha}([0,T]^2_{<};E)$ and its ${\mathscr{C}^\alpha([0,T];E)}$-norm does not exceed $\|G\|_{\mathscr{L}(E)}(K_{\alpha,0}\vee C_{\alpha,0,T})\|y\|_{{\mathscr Y}_{\alpha,\eta}(0,T)}$.

(ii) 
Applying Proposition \ref{prop:old_sew_map}, with $\omega=\lambda=0$, $\zeta=\alpha$ and $\rho=\eta+\alpha$, we conclude that the rough integral $\mathscr I_{Gy}(s,t)$ is well-defined for every $(s,t)\in[0,T]^2_<$ (see formula \eqref{def_int_rough}) and there exists a positive constant $\mathfrak C_1$ such that
\begin{align*}
\|{\mathscr I}_{Gy}-g_{G\check y}\|_{\mathscr C^{2\eta+\alpha}([0,T]^2_<;E)}
\leq \mathfrak C_1(1+[(x,\mathbb X)]_{\mathcal D^{\eta}([0,T])}^2)\|y\|_{{\mathscr Y}_{\alpha,\eta}(0,T)}.
\end{align*}
The proof is complete.
\end{proof}

Fix $X=(x,\mathbb X)\in\mathcal D^\eta([0,T])$ and $\psi\in E_{\eta+\alpha}$ for some $\alpha\in(1-2\eta,\eta)$. Further, let $(x_n,\mathbb X_n)_{n\in\N}$ and $h\in C^{2\eta}([0,T])$ be as in Proposition \ref{conv_rough_paths}, with $\eta'\in (\max\{(\eta+\alpha)/2,1-\alpha-\eta\},\eta)$, which implies that $\eta'>\alpha$.
Applying Theorem \ref{thm:ex_un_mild_solution_rough} and taking Remark \ref{rmk:es_sol_mild_2} into account, we conclude that there exist a unique mild solution $y\in\mathscr Y_{\alpha,\eta}(0,T)$ to the Cauchy problem \eqref{cauchy_prob_rough} 
and, for every $n\in\N$, a mild solution $y_n$ with $Gy_n$ as $SG$-derivative of the Cauchy problem 
\begin{align}
\label{rough_cauchy_probl_n}
\left\{
\begin{array}{ll}
dy(t)=Ay(t)dt+Gy(t)dx_n(t),    &  t\in(0,T], \vspace{1mm} \\
y(0)=\psi,  
\end{array}
\right.
\end{align}
and $y_n$ is given by the formula 
\begin{align}
y_n(t)=&S(t)\psi+S(t)G\psi(x_n(t)-x_n(0))+S(t)G^2\psi\mathbb{X}_n(0,t) \notag \\
&+S(t)G^2\psi(\delta_1h)(0,t)-M_{\delta_{S,2}g_{G^2y_n}}(0,t)
\label{def_y_n_appr}
\end{align}
for every $t\in [0,T]$ (see \eqref{mild_sol_form_espl}). In particular, $y_n$ is as smooth as $y$ is.

\begin{proposition}
\label{prop:conv_sol_reg_sol_rough}
The sequence $(y_n)_{n\in\N}$, defined in \eqref{def_y_n_appr}, converges to $y$ in $\mathscr Y_{\alpha,\eta}(0,T)$. Moreover,
$\mathscr{I}_{Gy_n}$ converges to ${\mathscr I}_{Gy}$ in $\mathscr{C}^{\eta'}([0,T]^2_{<};E)$.
\end{proposition}

\begin{proof}
By Proposition \ref{conv_rough_paths}, we can choose the sequence $(x_n,\mathbb{X}_n+\delta_1h)_{n\in\N}$ to be bounded in $C^{\eta}([0,T])\times\mathscr{C}^{2\eta}([0,T]^2_{<})$. Therefore, in the rest of the proof, we denote by $\kappa$ a positive constant such that
\begin{eqnarray*}
\|y\|_{\mathscr{Y}_{\alpha,\eta}(0,T)}+[x_n]_{C^{\eta}([0,T])}+\|\mathbb{X}_n+\delta_1h\|_{\mathscr{C}^{2\eta}([0,T]^2_<)}\le\kappa,\qquad\;\,n\in\N.
\end{eqnarray*}

We also denote by $\mathfrak{M}_j$ positive constants, which may depend on $T$ and $\kappa$, but are independent of $n$. Finally, to simplify the notation, we find it useful to set, for every $n\in\N$, $h_n=g_{G^2y_n}-g_{G^2y}$, 
\begin{eqnarray*}
\zeta_{n,\tau_1,\tau_2}=
\|y_n-y\|_{B([\tau_1,\tau_2];E_{\eta+\alpha})}+\|R^{Y_n}-R^Y\|_{\mathscr{C}^{\eta+\alpha}([\tau_1,\tau_2]^2_{<};E_{\alpha})}
\end{eqnarray*}
and
\begin{eqnarray*}
\Theta_n=[x-x_n]_{C^{\eta'}([0,T])}+\|\mathbb{X}_n+\delta_1h-\mathbb{X}\|_{{\mathscr C}^{2\eta'}([0,T]^2_{<})}.
\end{eqnarray*}

Since it is rather long, we split the rest of the proof into different steps. In Steps 1 to 3, we prove the convergence of $y_n$ to $y$ in $\mathscr{Y}_{\alpha,\eta}(0,T)$ and in Step 4, we prove the convergence of ${\mathscr I}_{Gy_n}$ to ${\mathscr I}_{Gy}$.

{\it Step $1$}. Here, we prove that there exists $\overline T\leq T\wedge 1$ such that $\zeta_{n,0,\overline{T}}$ converges to zero as $n$ tends to $0$. From Remark \ref{rmk:lin_M}(i) and \eqref{mild_sol_form_espl} we infer that
\begin{align}
y_n(t)-y(t)
= & S(t)G\psi((x_n-x)(t)-(x_n-x)(0))\notag\\
&+S(t)G^2\psi\big (\mathbb X_n(0,t)+(\delta_1h)(0,t)-\mathbb{X}(0,t)\big )
-M_{\delta_{S,2}h_n}(0,t)
\label{y_n-y_1}
\end{align}
for every $t\in[0,T]$ and $n\in\N$.

According to \eqref{S_incr_A}, we can split
$\delta_{S,2}h_n=
\mathscr{K}_{1,n}+\mathscr{K}_{2,n}$ for
every $n\in\N$, where
\begin{align*}
&\mathscr{K}_{1,n}(s,t,u)=R^{GY}(s,t)(x(u)-x(t))-R^{GY_n}(s,t)(x_n(u)-x_n(t)),
\\[1mm]
&\mathscr{K}_{2,n}(s,t,u)=(\hat\delta_1G^2y)(s,t)\mathbb{X}(t,u)-(\hat\delta_1G^2y_n)(s,t)(\mathbb{X}_n(t,u)
+(\delta_1h)(t,u))
\end{align*}
for every $(s,t,u)\in [0,\overline T]^3_{<}$.
We estimate separately the functions $\mathscr{K}_{1,n}$ and $\mathscr{K}_{2,n}$.

Taking advantage of \eqref{forma_rem_GY}, we can write
\begin{align*}
\mathscr{K}_{1,n}(s,t,u)=&
R^{GY}(s,t)((x-x_n)(u)-(x-x_n)(t))\\
&+(R^{GY}(s,t)-R^{GY_n}(s,t))(x_n(u)-x_n(t))\\
=&R^{GY}(s,t)((x-x_n)(u)-(x-x_n)(t))\\
&+[G,\mathfrak{a}(s,t)]G(y(s)-y_n(s))(x_n(t)-x_n(s))(x_n(u)-x_n(t))\\
&+[G,\mathfrak{a}(s,t)]Gy(s)((x-x_n)(t)-(x-x_n)(s))(x_n(u)-x_n(t))\\
&+[G,\mathfrak{a}(s,t)](y(s)-y_n(s))(x_n(u)-x_n(t))\\
&+G(R^Y(s,t)-R^{Y_n}(s,t))(x_n(u)-x_n(t))
\end{align*}
for every $(s,t,u)\in [0,\overline T]^3_{<}$. Taking also \eqref{stima_S_der_Gy} into account, we can estimate
\begin{align*}
&|\mathscr{K}_{1,n}(s,t,u)|_E \notag \\
\le & \kappa
\big (C_{\alpha,0}\mathfrak{g}_{0,\alpha}^2\kappa
+C_{\eta+\alpha,0}\mathfrak{g}_{0,\eta+\alpha}
+ K_{\alpha,0}\|G\|_{\mathscr{L}(E)}\big )
[x-x_n]_{C^{\eta'}([0,T])}|u-s|^{\eta'+\eta+\alpha} \notag \\
&+C_{\eta+\alpha,0}\mathfrak{g}_{0,\eta+\alpha}^2\|y_n-y\|_{B([0,\overline T];E_{\eta+\alpha})}\kappa^2|u-s|^{3\eta+\alpha} \notag \\
&+C_{\eta+\alpha,0}\mathfrak{g}_{0,\eta+\alpha}^2\kappa^2[x-x_n]_{C^{\eta'}([0,T])}|u-s|^{2\eta+\eta'+\alpha} \notag \\
&+C_{\eta+\alpha,0}\mathfrak{g}_{0,\eta+\alpha}\|y_n-y\|_{B([0,\overline T];E_{\eta+\alpha})}\kappa|u-s|^{2\eta+\alpha} \notag \\
&+\mathfrak{g}_0\|R^Y-R^{Y_n}\|_{\mathscr{C}^{\eta+\alpha}([0,\overline T]^2_<;E)}\kappa |u-s|^{2\eta+\alpha},
\end{align*}
so that $\mathscr{K}_{1,n}$ belongs
to $\mathscr{C}^{\eta+\eta'+\alpha}([0,\overline{T}]^3_{<};E)$ and
\begin{align}
\|\mathscr{K}_{1,n}\|_{\mathscr{C}^{\eta+\eta'+\alpha}([0,\overline{T}]^3_{<};E)}
\le\mathfrak{M}_0\big ([x-x_n]_{C^{\eta'}([0,T])}+\zeta_{n,0,\overline{T}}\big ).
\label{stima_K_1}
\end{align}


Similarly, for every $(s,t,u)\in[0,\overline T]^3_<$ we can estimate
\begin{align*}
|\mathscr{K}_{2,n}(s,t,u)|_{E}\le &
|(\hat\delta_1G^2y)(s,t)|_{E}\|\mathbb{X}_n+\delta_1h-\mathbb{X}\|_{\mathscr C^{2\eta'}([0,T]^2_{<})}|u-t|^{2\eta'}\\
&+|(\hat\delta_1G^2y)(s,t)-(\hat\delta_1G^2y_n)(s,t)|_{E}\|\mathbb{X}_n+\delta_1h\|_{\mathscr C^{2\eta}([0,T]^2_{<})}|u-t|^{2\eta}.
\end{align*}

Since $(\hat\delta_1G^2y)(s,t)=G^2(\hat\delta_1y)(s,t)+[G^2,\mathfrak{a}(s,t)]y(s)$ for every $(s,t)\in[0,\overline T]^2_<$, 
taking Remark \ref{rmk:reg_hat_delta_1_y} into account, which shows that $\hat\delta_1y\in\mathscr{C}^{\eta}([0,T]^2_{<};E_{\alpha})$ and $\|\hat\delta_
1y\|_{\mathscr{C}^{\eta}([0,T]^2_{<};E_{\alpha})}\le (L_{\alpha,\alpha}[x]_{C^{\eta}([0,T])}+T^{\alpha})\|y\|_{{\mathscr Y}_{\alpha,\eta}(0,T)}$,  we obtain
\begin{align*}
|(\hat\delta_1G^2y)(s,t)|_E\le &K_{\alpha,0}\mathfrak{g}_{\alpha}^2
(L_{\alpha,\alpha}[x]_{C^{\eta}([0,T])}+T^{\alpha})\|y\|_{{\mathscr Y}_{\alpha,\eta}(0,T)}|t-s|^{\eta}\notag\\
&+C_{\eta+\alpha,0}\mathfrak{g}_{0,\eta+\alpha}^2\|y\|_{B([0,T];E_{\eta+\alpha})}|t-s|^{\eta+\alpha}.
\end{align*}

Similarly, using the above expression on $\hat\delta_1G^2y$ and the definition of SG-derivative, we can split
\begin{align*}
&(\hat\delta_1G^2y)(s,t)-(\hat\delta_1G^2y_n)(s,t)\notag\\
=&G^2S(t-s)[Gy(s)((x-x_n)(t)-(x-x_n)(s))-
G(y_n(s)-y(s))(x_n(t)-x_n(s))]\notag\\
&+G^2R^Y(s,t)-G^2R^{Y_n}(s,t)
+[G^2,\mathfrak a(s,t)](y(s)-y_n(s))
\end{align*}
for every $(s,t)\in[0,\overline T]^2_<$, and estimate
\begin{align*}
&|(\hat\delta_1G^2y)(s,t)-(\hat\delta_1G^2y_n)(s,t)|_E\\
\le & K_{\eta+\alpha,0}\mathfrak{g}_0^3
L_{0,0}\|y\|_{B([0,T];E_{\eta+\alpha})}
[x_n-x]_{C^{\eta'}([0,T])}|u-s|^{\eta'}\\
&+K_{\eta+\alpha,0}\mathfrak{g}_0^3
L_{0,0}\|y_n-y\|_{B([0,\overline T];E_{\eta+\alpha})}
[x_n]_{C^{\eta}([0,T])}|u-s|^{\eta}\\
&+\mathfrak{g}_0^2\|R^Y-R^{Y_n}\|_{\mathscr{C}^{\eta+\alpha}([0,\overline T]^2_{<};E)}|u-s|^{\eta+\alpha}\\
&+C_{\eta+\alpha,0}\mathfrak{g}_{0,\eta+\alpha}^2\|y_n-y\|_{B([0,\overline T];E_{\eta+\alpha})}|u-s|^{\eta+\alpha}
\end{align*}
for every $(s,t)\in[0,\overline{T}]^2_<$ and $n\in\N$. Hence, $\mathscr{K}_{2,n}\in\mathscr{C}^{2\eta'+\eta}([0,\overline{T}]^3_{<};E)$ and
\begin{align}
\|\mathscr{K}_{2,n}\|_{\mathscr{C}^{2\eta'+\eta}([0,\overline{T}]^3_{<};E)}
\le
\mathfrak{M}_1\big (\Theta_n+\zeta_{n,0,\overline{T}}\big ).
\label{stima_K_2}
\end{align}

Recalling that
$\eta'\in (\max\{(\eta+\alpha)/2,1-\alpha-\eta\},\eta)$ and $\|\cdot\|_{\mathscr{C}^{\eta+\eta'+\alpha}([0,T]^3_{<};E)}\le T^{\eta'-\alpha}\|\cdot\|_{\mathscr{C}^{2\eta'+\eta}([0,T]^3_{<};E)}$, from \eqref{stima_K_1} and \eqref{stima_K_2} it follows that
\begin{align}
\|\delta_{S,2}h_n\|_{\mathscr{C}^{\eta+\eta'+\alpha}([0,T]^3_{<};E)}\leq \mathfrak M_2\big (
\Theta_n+\zeta_{n,0,\overline T}\big )
\label{stima_delta_2_h_n-0}
\end{align}
for every $n\in\N$. 
From estimates \eqref{Mg}, with $g=\delta_{S,2}h_n$, $\mu=\eta+\eta'+\alpha$, $\beta=0$, $a=0$ and $b=\overline T$, and \eqref{stima_delta_2_h_n-0}, we infer that for every $\varepsilon\in[0,1)$ there exists a positive constant $C$, depending on $T$ and $\varepsilon$, such that
\begin{align}
\|M_{\delta_{S,2}h_n}\|_{\mathscr C^{\eta+\eta'+\alpha-\varepsilon}([0,\overline T]^2_<;E_{\varepsilon})}
\leq C \mathfrak M_2\big(
\Theta_n+\zeta_{n,0,\overline T}\big)
\label{stima_M_delta_2_h_n}
\end{align}
for every $n\in\N$. Formula \eqref{y_n-y_1} and estimate \eqref{stima_M_delta_2_h_n}, with $\varepsilon=\eta+\alpha$, give 
\begin{align}
\|y_n-y\|_{B([0,\overline T];E_{\eta+\alpha})}
\leq & \mathfrak M_3\big(\Theta_n
+\|R^{Y}-R^{Y_n}\|_{\mathscr C^{\eta+\alpha}([0,\overline T]^2_<;E_{\alpha})}\big)\notag\\
&+\mathfrak M_3\overline T^{\eta'}\|y_n-y\|_{B([0,\overline T];E_{\eta+\alpha})}
\label{stima_dif_y_n-y_norma_eta+alpha},
\end{align}
where we have estimated $|\psi|_{E_{\eta+\alpha}}$
from above by $\|y\|_{B([0,T];E_{\eta+\alpha})}$ to obtain a constant $\mathfrak{M}_3$, which is independent of $\psi$.

Now we estimate $\|R^{Y}-R^{Y_n}\|_{\mathscr C^{\eta+\alpha}([0,\overline T]^2_<;E_{\alpha})}$. Taking \eqref{rem-gamma} into account, we can write
\begin{align*}
R^{Y}(s,t)-R^{Y_n}(s,t)
= & S(t-s)G^2y(s)(\mathbb X(s,t)-\mathbb X_n(s,t)-(\delta_1h)(s,t)) \\
&+ S(t-s)G^2(y(s)-y_n(s))(\mathbb X_n(s,t)+(\delta_1h)(s,t))\\
&+M_{\delta_{S,2}h_n}(s,t)
\end{align*}
for every $n\in\N$ and $(s,t)\in[0,\overline T]^2_<$. From \eqref{stima_M_delta_2_h_n}, with $\varepsilon=\alpha$, 
we infer that
\begin{align*}
& |R^{Y_n}(s,t)-R^{Y}(s,t)|_{E_{\alpha}} \\
\leq & L_{\alpha,\alpha}\|G^2\|_{\mathscr L(E_\alpha)}K_{\eta+\alpha,\alpha}\|y\|_{B([0,T];E_{\eta+\alpha})}\|\mathbb X_n+\delta_1h-\mathbb X\|_{\mathscr C^{2\eta'}([0,T]^2_<)}|t-s|^{2\eta'} \\
& + L_{\alpha,\alpha}\|G^2\|_{\mathscr L(E_\alpha)}K_{\eta+\alpha,\alpha}\|y_n-y\|_{B([0,\overline T];E_{\eta+\alpha})}\|\mathbb X_n+\delta_1h\|_{\mathscr C^{2\eta}([0,T]^2_<)}|t-s|^{2\eta} \\
& + \|M_{\delta_{S,2}h_n}\|_{\mathscr C^{\eta+\eta'}([0,\overline T]^2_<;E_{\alpha})}|t-s|^{\eta+\eta'}
\end{align*}
for every $n\in\N$ and $(s,t)\in[0,\overline T]^2_<$. Recalling that $\eta'>(\eta+\alpha)/2>\alpha$, we conclude that
\begin{align}
\|R^{Y_n}-R^{Y}\|_{\mathscr C^{\eta+\alpha}([0,T]^2_<;E_{\alpha})}
\leq & \mathfrak M_4\Theta_n+\mathfrak M_4\overline T^{\eta'-\alpha}\zeta_{n,0,\overline{T}}
\label{stima_R_Y-R_Yn}
\end{align}
for every $n\in\N$.
From \eqref{stima_dif_y_n-y_norma_eta+alpha} and \eqref{stima_R_Y-R_Yn}, and recalling that $\overline T\leq 1$, we infer that, if $2[\mathfrak M_3(\mathfrak{M}_4\overline{T}^{\eta'-\alpha}+\overline{T}^{\eta'})+\mathfrak M_4\overline T^{\eta'-\alpha}]\leq 1$, then for every $n\in\N$ we get
\begin{align*}
& \zeta_{n,0,\overline T}\leq 2(\mathfrak M_3+\mathfrak M_4+\mathfrak M_3\mathfrak M_4)\Theta_n.
\end{align*}
Letting $n$ tend to infinity, we conclude that $\zeta_{n,0,\overline{T}}$ converges to $0$ as $n$ tends to $\infty$.    

{\em Step 2}. Here, based on Step 1, we prove that $y_n$ converges to $y$ in $\mathscr{Y}_{\alpha,\eta}(0,\overline{T})$. For this purpose, we begin by observing that, since $\check y_n=Gy_n$ and $\check y=Gy$, from Step 1 it follows that
$\check y_n$ converges to $\check y$ in $B([0,\overline{T}];E_{\eta+\alpha})$ (and, consequently, in $B([0,\overline{T}];E_{\alpha})$) as $n$ tends to $\infty$.

As far as the difference $\hat\delta_1\check y_n-\hat\delta_1\check y$ is concerned, observing that
\begin{align*}
(\hat\delta_1y)(s,t)=&\mathscr I_{SGy}(s,t)\\
=&S(t-s)Gy(s)(x(t)-x(s))+S(t-s)G^2y(s)\mathbb{X}(s,t)
-M_{\delta_{S.2}g}(s,t)
\end{align*}
for every $(s,t)\in [0,T]^2_{<}$
(and a similar formula holds true for $y_n$, just replacing $x$ and $\mathbb{X}$ with $x_n$ and $\mathbb{X}_n+\delta_1h$, respectively), we get
\begin{align*}
&(\hat\delta_1\check y_n)(s,t)
-(\hat\delta_1\check y)(s,t)\\
= & G(\hat\delta_1y_n)(s,t)+[G,\mathfrak a(s,t)]y_n(s)-G(\hat\delta_1y)(s,t)-[G,\mathfrak a(s,t)]y(s) \\
=&GS(t-s)(Gy_n(s)-Gy(s))(x_n(t)-x_n(s))+[G,\mathfrak{a}(s,t)](y_n-y)(s)  
\\
&+GS(t-s)Gy(s)((x_n-x)(t)-(x_n-x)(s))\\
&+GS(t-s)G^2(y_n(s)-y(s))(\mathbb{X}_n(s,t)+(\delta_1h)(s,t))\\
&+GS(t-s)G^2y(s)(\mathbb{X}_n(s,t)+(\delta_1h)(s,t)-\mathbb{X}(s,t))-
GM_{\delta_{S,2}h_n(s,t)}
\end{align*}
for every $(s,t)\in[0,\overline T]^2_<$ and $n\in\N$. From \eqref{stima_M_delta_2_h_n} with $\eta=\alpha$ it follows that $\hat\delta_1\check{y}_n-\hat\delta_1\check{y}$ belongs to $\mathscr{C}^{\alpha}([0,\overline{T}]^2_{<};E_{\alpha})$ and
\begin{align}
\|\hat\delta_1\check{y}_n-\hat\delta_1\check{y}\|_{\mathscr C^{\alpha}([0,\overline T]^2_{<};E_{\alpha})}\leq \mathfrak M_5
(\Theta_n+\zeta_{n,0,\overline{T}}\big),\qquad\;\,n\in\N.
\label{stima_diff_Gy_n-Gy}
\end{align}
From Step 1 and \eqref{stima_diff_Gy_n-Gy}, we conclude that
$y_n-y$ vanishes in ${\mathscr Y_{\alpha,\eta}(0,\overline T)}$ as $n$ tends to $\infty$. 

{\it Step 3}. If $\overline T= T$ then we are done. Otherwise, we consider the interval $[\overline T,T_1]$, where $T_1=\min\{T,2\overline T\}$
and, taking \eqref{def_int_conv_rough} into account, we write
\begin{align*}
y_n-y=&S(\cdot-\overline{T})(y_n(\overline{T})-y(\overline T))+S(\cdot-s)(Gy_n(\overline{T})-Gy(\overline{T}))(x_n-x_n(\overline{T}))\\
&+S(\cdot-s)Gy(\overline{T})(x_n-x-(x_n-x)(\overline{T}))\\
&+S(\cdot-s)(G^2y_n(\overline{T})-G^2y(\overline{T}))(\mathbb{X}_n(\overline{T},\cdot)+\delta_1h(\overline{T},\cdot))\\
&+S(\cdot-s)G^2y(\overline{T})(\mathbb{X}_n(\overline{T},\cdot)+\delta_1h(\overline{T},\cdot)-\mathbb{X}(\overline{T},\cdot))
-M_{\delta_{S,2}h_n}(\overline{T},t)=:\sum_{j=1}^6I_j
\end{align*}
for every $t\in [\overline{T},T_1]$ and every $n\in\N$.

The sum $I_3+I_5+I_6$ can be estimated arguing as in Step 1 and we can show that
\begin{align*}
\|I_3+I_5+I_6\|_{B([\overline{T},T_1];E_{\eta+\alpha})}
\le \mathfrak{M}_3(1+\mathfrak{M}_4)\Theta_n+\mathfrak{M}_3(\mathfrak{M}_4\overline{T}^{\eta'-\alpha}+\overline{T}^{\eta'})\zeta_{n,\overline T,T_1}
\end{align*}
and
\begin{align*}
\|R^{Y_n}-R^{Y}\|_{\mathscr C^{\eta+\alpha}([\overline{T},T_1]^2_<;E_{\alpha})}
\leq & \mathfrak M_4\Theta_n+\mathfrak M_4\overline T^{\eta'-\alpha}\zeta_{n,\overline{T},T_1}
\end{align*}
(see \eqref{stima_dif_y_n-y_norma_eta+alpha}
and \eqref{stima_R_Y-R_Yn}).

On the other hand, the sum $I_1+I_2+I_4$ is immediate to estimate and we find out that
\begin{align*}
\|I_1+I_2+I_4\|_{B([\overline{T},T_1];E_{\eta+\alpha})}
\le \mathfrak{M}_6\|y_n-y\|_{B([0,\overline T];E_{\eta+\alpha})}.
\end{align*}

Combining these last three estimates we can infer that
\begin{eqnarray*}
\zeta_{n,\overline{T},T_1}\le 2(\mathfrak{M}_3+\mathfrak{M}_4+\mathfrak{M}_3\mathfrak{M}_4)\Theta_n+\mathfrak{M}_6\zeta_{n,0,\overline{T}},\qquad\;\,n\in\N,
\end{eqnarray*}
and from Step 1,  we conclude that $\zeta_{n,\overline{T},T_1}$ vanishes as $n$ tends to $\infty$. As a byproduct, we obtain that
$y_n$ converges to $y$ in $B([0,T_1];E_{\eta+\alpha})$. 

To show that $\zeta_{n,0,T_1}$ vanishes as $n$ tends to $ \infty$, we need to prove that $R^{Y_n}$ converges to $R^Y$ in $\mathscr{C}^{\eta+\alpha}([0,T_1]^2_{<};E_{\alpha})$. 
For this purpose, we observe that
\begin{align*}
R^{Y}(s,t)-R^{Y}(\overline{T},t)=&(S(t-\overline{T})y(\overline{T})-S(t-s)y(s))-S(t-s)Gy(s)(x(\overline{T})-x(s))\\
&+(S(t-\overline T)Gy(\overline T)-S(t-s)Gy(s))
(x(t)-x(\overline T))\\
=&S(t-\overline{T})(\hat\delta_1y)(s,\overline{T})-S(t-s)Gy(s)(x(\overline{T})-x(s))\\
&+S(t-\overline{T})(\hat\delta_1Gy)(s,\overline{T})(x(t)-x(\overline{T}))\\
=&S(t-\overline{T})R^{Y}(s,\overline{T})+S(t-\overline{T})(\hat\delta_1Gy)(s,\overline{T})(x(t)-x(\overline{T}))
\end{align*}
for every $(s,t)\in [0,T_1]$, with $s<\overline T<t$ and $n\in\N$,
and the same formula holds true with $(y,Y,x)$ being replaced by $(y_n,Y_n,x_n)$.
Since $R^{Y_n}$ converges to $R^{Y}$ in $\mathscr C^{\eta+\alpha}([0,\overline{T}]^2_<;E_{\alpha})$ and in
$\mathscr C^{\eta+\alpha}([\overline{T},T_1]^2_<;E_{\alpha})$, we can adapt the argument in the second part of Step 1 of the proof of Theorem \ref{thm:ex_un_mild_solution_rough} to show that 
\begin{align*}
\|R^{Y_n}-R^Y\|_{\mathscr{C}^{\eta+\alpha}([0,T_1]^2_{<};E_{\alpha})}\le \mathfrak{M}_7(\Theta_n+
\zeta_{n,0,\overline{T}}+\zeta_{n,\overline{T},T_1}), \qquad n\in\N, 
\end{align*}
from which the claimed convergence of $R^{Y_n}$ to $R^Y$ follows at once.

We have so proved that 
$\zeta_{n,0,T_1}$ vanishes as $n$ tends to $\infty$.
Now, arguing as in Step 2, we infer that $y_n$ converges to $y$
in ${\mathscr{Y}_{\alpha,\eta}(0,T_1)}$ as $n$ tends to $\infty$. If $T_1=T$ then we are done, otherwise we split $[0,T]$ into intervals of the form $[(j-1)\overline T,j\overline T]$ and iterate the above arguments. In a finite number of steps we complete the proof of the convergence of $y_n$ to $y$ in $\mathscr{Y}_{\alpha,\eta}(0,T)$.

{\em Step 4}. From Lemma \ref{lem:esistenza_int_rough} we already know that $\mathscr I_{Gy}$ and $\mathscr I_{Gy_n}$ ($n\in\N$) are well-defined and, for every $n\in\N$ and $(s,t)\in[0,T]^2_<$, we get
\begin{align}
& \mathscr I_{Gy_n}(s,t)
-\mathscr I_{Gy}(s,t) \notag \\
=& G(y_n(s)-y(s))(x_n(t)-x_n(s))+Gy(s)((x_n-x)(t)-(x_n-x)(s)) \notag \\
& + G^2(y_n(s)-y(s))(\mathbb X_n(s,t)+(\delta_1h)(s,t))\notag \\
&+G^2y(s)(\mathbb X_n(s,t)+(\delta_1h)(s,t)-\mathbb X(s,t))  -M_{\delta_2h_n}(s,t).
\label{diff_int_1}
\end{align}
From the definition of $h_n$, it follows that
$\delta_2h_n
=\widetilde{\mathscr K}_{1,n}+\widetilde{\mathscr K}_{2,n}$, where the terms 
$\widetilde{\mathscr K}_{1,n}$ and $\widetilde{\mathscr K}_{2,n}$ are defined as
the corresponding terms $\mathscr{K}_{1,n}$ and $\mathscr{K}_{2,n}$ in Step 1, with $[G,\mathfrak{a}]$, $R^{GY}$ and $\hat\delta_1$ being everywhere replaced by $G\mathfrak{a}$, $\widetilde R^{GY}$ (see \eqref{gub_deriv_Y}) and  $\delta_1$, respectively. Note that 
$G\mathfrak{a}$ has the same smoothing properties as the commutator $[G,\mathfrak{a}]$ and 
from Remark \ref{rmk:reg_hat_delta_1_y}
we can infer that $\delta_1y\in \mathscr{C}^{\eta}([0,T]^2_{<};E)$ and its $\mathscr{C}^{\eta}([0,T]^2_{<};E)$-norm can be bounded from above by
$K_{\alpha,0}\kappa(1+C_{\eta+\alpha,\eta})$. Hence, taking also \eqref{intro-mario} into account and arguing as in Step 1, we can easily show that
\begin{align}
\|\delta_2h_n\|_{\mathscr C^{\eta+\eta'+\alpha}([0,T]^3_<;E)}
\leq & \mathfrak M_8\big(\Theta_n+\zeta_{n,0,T}\big ),\qquad\;\,n\in\N.
\label{stima_delta_2_h_n}
\end{align}

From Remark \ref{rmk:lin_M}(ii), with $S(t)=I$ for every $t\geq0$, $g=\delta_2h_n$, $\mu=\eta+\eta'+\alpha$, $\beta=0$, $a=0$ and $b=T$, and \eqref{stima_delta_2_h_n}, we infer that there exists a positive constant $C$ such that, for every $n\in\N$,
\begin{align}
\|M_{\delta_2h_n}\|_{\mathscr C^{\eta+\eta'+\alpha}([0,T]^2_<;E)}
\leq & \mathfrak M_9\big(\Theta_n+\zeta_{n,0,T}\big ),\qquad\;\,n\in\N.
\label{stima_int_M_delta_2_h_n}
\end{align}
From \eqref{diff_int_1} and \eqref{stima_int_M_delta_2_h_n} it is immediate to conclude that
\begin{align}
\|\mathscr I_{Gy_n}-\mathscr I_{Gy}\|_{\mathscr C^{\eta'}([0,T]^2_<;E)}
\leq & \mathfrak M_{10}\big(\Theta_n+\zeta_{n,0,T}\big),\qquad\;\,n\in\N.
\label{stima_diff_int_finale}
\end{align}
Letting $n$ tend to infinity in \eqref{stima_diff_int_finale}, we complete the proof.
\end{proof}

To prove the next result, we assume an additional condition, which is clearly satisfied if $(S(t))_{t\ge 0}$ is a strongly continuous or an analytic semigroup. 
\begin{hypothesis}
For every $t\geq0$ and $x\in E$, $\displaystyle \int_0^t S(r)xdr$ belongs to $D(A)$ and 
\begin{align}
\label{midnight_formula}
A\int_0^t S(r)xdr=S(t)x-x.  
\end{align}
\end{hypothesis}

\begin{theorem} 
\label{thm-4.8}
If $(x,\mathbb X)$ is a geometric rough path and $y$ is the unique mild solution to \eqref{cauchy_prob_rough} belonging to $\mathscr Y_{\alpha,\eta}(0,T)$, then formula \eqref{rappr_int_sol} holds true.
\end{theorem}

\begin{proof}
Let $y_n$ be the  mild solution to \eqref{rough_cauchy_probl_n} given by \eqref{def_y_n_appr}. From Remark \ref{rmk:un_rough_young_int}(ii), which shows that the Young integral coincides with the rough integral when both are well-defined, and recalling that the Young integral is nothing but the classical integral when the path is smooth, it follows that
\begin{align*}
y_n(t)=S(t)\psi+\int_0^tS(t-r)Gy_n(r)x_n'(r)dr, \qquad t\in[0,T]. 
\end{align*}

Since $y_n$ is a mild solution in a classical sense and the semigroup $(S(t))_{t\geq0}$ satisfies \eqref{midnight_formula}, computations similar to those in \cite[Proposition 4.1.5]{LU95} imply that
\begin{align*}
y_n(t)=\psi+\int_0^tAy_n(s)ds+\int_0^tGy_n(s)x_n'(s)ds, \qquad t\in[0,T].
\end{align*}

Similarly, taking $S(t)=I$ for every $t\ge 0$, we can easily prove that
\begin{align*}
\mathscr I_{Gy_n}(s,t)=\int_s^tGy_n(r)x'_n(r)dr, \qquad (s,t)\in[0,T]^2_<.    
\end{align*}

We are almost done. Indeed, we have proved that
\begin{align*}
A\int_0^ty_n(s)ds=
\int_0^tAy_n(s)ds
= y_n(t)-\psi-\mathscr I_{Gy_n}(0,t), \qquad t\in[0,T].    
\end{align*}
Let us fix $t>0$. From Proposition \ref{prop:conv_sol_reg_sol_rough} we infer that
$\int_0^ty_n(s)ds$ and
$A\int_0^ty_n(s)ds$ converge, respectively, to $\int_0^ty(s)ds$ and $y(t)-\psi-\mathscr I_{Gy}(0,t)$ as $n$ tends to $\infty$,
where both the convergences are meant in $E$. Since $A$ is a closed operator, we conclude that $\int_0^ty(s)ds\in D(A)$ and
\begin{align*}
A\int_0^ty(s)ds=y(s)-\psi-\mathscr I_{Gy}(0,t),\qquad\;\,{\rm i.e.}\;\,
y(t)=\psi+A\int_0^ty(s)ds+\mathscr I_{Gy}(0,t). 
\end{align*}

To conclude the proof, it suffices to recall that estimate \eqref{estim-1} shows that  $Ay$ belongs to $L^1((0,T))$, so that $A$ commutes with the integral of $y$ over the interval $(0,t)$.
\end{proof}

\section{It\^{o} formula for mild solutions to
(\ref{cauchy_prob_rough})}
\label{sect-5}
In this section, based on the integral representation of the solution to equation \eqref{cauchy_prob_rough} in Theorem \ref{thm-4.8}, we prove a chain rule formula. 
We recall that, for every pair of Banach spaces $E$ and $F$, and $\alpha,\gamma\in [0,1)$, $C^{\alpha,\gamma}([0,T]\times E;F)$ denotes the space of all bounded functions $f:[0,T]\times E\to F$ such that $f(\cdot,x)$ is $\alpha$-H\"older continuous in $[0,T]$, uniformly with respect to $x\in E$, and $f(t,\cdot)$ is $\gamma$-H\"older continuous in $E$, uniformly with respect to $t\in [0,T]$. It is a Banach space when endowed with the norm $\|f\|_{C^{\alpha,\gamma}([0,T]\times E;F)}=\sup_{x\in E}\|f(\cdot,x)\|_{C^{\alpha}([0,T];F)}+\sup_{t\in [0,T]}[f(t,\cdot)]_{C^{\alpha}_b(E;F)}$. Moreover, $C^{1,2}([0,T]\times E;E)$ is the space of all functions $f:[0,T]\times E\to E$, which are once differentiable in $[0,T]\times E$ with respect to the $t$ variable, with derivative $f_t$ which is continuous in $[0,T]\times E$, and such that $f(t,\cdot)$ is twice Fr\'echet differentiable in $E$ with Fr\'echet derivatives $f_x$ and $f_{xx}$, which are continuous in $[0,T]\times E$. It is normed by setting
$\|f\|_{C^{1,2}([0,T]\times E)}=\|f\|_{\infty}+\|f_t\|_{\infty}+\|f_x\|_{B([0,T]\times E;\mathscr{L}(E))}+\|f_{xx}\|_{B([0,T]\times E;{\rm Bil}(E))}$, where ${\rm Bil}(E)$ is the set of all the bilinear operators over $E$.

\begin{theorem}
Let $F\in C^{1,2}([0,T]\times E;E)$ be such that
$F_x\in C^{\sigma,0}([0,T]\times E;\mathscr{L}(E))$,
$F_{xx}\in C^{\lambda,\gamma}([0,T]\times E;{\rm Bil}(E))$, for some $\sigma\in (1-\eta,1)$, $\lambda\in (1-2\eta,1)$ and $\gamma$ such that $(2+\gamma)\eta>1$. Finally, let $y$ be the solution to equation \eqref{cauchy_prob_rough} provided by Theorem $\ref{thm:ex_un_mild_solution_rough}$.
Then, for every $(s,t)\in [0,T]^2_{<}$ the following It\^{o} formula holds true:
\begin{align}
F(t,y(t))-F(s,y(s))=&\int_s^tF_t(u,y(u))du
+\int_s^t\langle F_x(u,y(u)),Ay(u)\rangle du \notag \\
&
+\int_s^t\langle F_x(u,y(u)),Gy(u)\rangle dx(u).
\label{ito_formula}
\end{align}
\end{theorem}

\begin{proof}
We fix $s,t\in (0,T]$, with $s<t$, and, for every $n\in\N$, a partition $\Pi_n(s,t)=\{s=s_0^n<\ldots<s^n_{m_n}=t\}$ of the interval $[s,t]$, with a mesh whose magnitude tends to zero as $n$ tends to $\infty$, and evaluate the difference $F(t,y(t))-F(s,y(s))$. A simple computation shows that
\begin{align*}
&F(t,y(t))-F(s,y(s))\\
=&\sum_{j=1}^{m_n}(F(s_j^n,y(s_j^n))-F(s_{j-1}^n,y(s_{j-1}^n))\\
=&\sum_{j=1}^{m_n}[F(s_j^n,y(s_j^n))-F(s_{j-1}^n,y(s^n_j))]
+\sum_{j=1}^{m_n}[F(s_{j-1}^n,y(s_j^n))-F(s_{j-1}^n,y(s_{j-1}^n))]\\
=&\sum_{j=1}^{m_n}[F_t(\widetilde s_j^n,y(s^n_j))
-F_t(s_j^n,y(s^n_j))](s^n_j-s^n_{j-1})
+\sum_{j=1}^{m_n}F_t(s_j^n,y(s^n_j))(s^n_j-s^n_{j-1})\\
&+\sum_{j=1}^{m_n}\langle F_x(s_{j-1}^n,y(s^n_{j-1})),
y(s^n_j)-y(s^n_{j-1})\rangle\\
&+\frac{1}{2}\sum_{j=1}^{m_n}\langle [F_{xx}(s_{j-1}^n,\widetilde y^n_j)-F_{xx}(s_{j-1}^n,y(s^n_{j-1}))](y(s^n_j)-y(s^n_{j-1}),y(s^n_j)-y(s^n_{j-1})\rangle\\
&+\frac{1}{2}\sum_{j=1}^{m_n}\langle F_{xx}(s_{j-1}^n,y(s^n_{j-1}))(y(s^n_j)-y(s^n_{j-1}),y(s^n_j)-y(s^n_{j-1})\rangle=:\sum_{k=1}^5I_k^n,
\end{align*}
where, for every $j=1,\ldots,m_n$, $\widetilde s^n_j$ is a suitable point in the interval $(s_{j-1}^n,s_j^n)$ and $\widetilde y^n_j=\sigma y^n_{j-1}+(1-\sigma)y^n_j$ for a suitable $\sigma\in (0,1)$.

We consider the terms $I_k^n$ ($k=1,\ldots,5$) separately.
Since $F_t$ is continuous on $[0,T]\times E$, it is uniformly continuous on $[0,T]\times y([s,t])$. Hence, for every $\varepsilon>0$, we can determine $n_0\in\N$ such that $|F_t(\widetilde s_j^n,y(s^n_j))
-F_t(s_j^n,y(s^n_j))|_E\le\varepsilon$ for every $j=0,\ldots,m_n$, if 
$n\ge n_0$. Hence,
\begin{align*}
|I_1^n|_E\le \sum_{j=1}^{m_n}|F_t(\widetilde s_j^n,y(s^n_j))
-F_t(s_j^n,y(s^n_j))|_E(s^n_j-s^n_{j-1})
\le\varepsilon (t-s)
\end{align*}
for every $n\ge n_0$, so that
\begin{equation}
\lim_{n\to \infty}I_1^n=0.
\label{lim-I1}
\end{equation}

As far as $I_2^n$ is concerned, we observe that, since the function $F_t$ is continuous in $[0,T]\times E$ and $y$ is continuous in $E$, it follows that
\begin{eqnarray}
\lim_{n\to +\infty}I_2^n=\int_s^tF_t(u,y(u))du.
\label{lim-I2}
\end{eqnarray}

We skip for a while the term $I_3^n$ and pass to consider $I_4^n$. Since $F_{xx}$ belongs to $C^{0,\gamma}([0,T];{\rm Bil}(E))$, we can estimate
\begin{align*}
|I_4^n|_E\le &\frac{1}{2}[F_{xx}]_{C^{0,\gamma}([0,T];{\rm Bil}(E))}\sum_{j=1}^{m_n}|y(s^n_j)-y(s^n_{j-1})|^{2+\gamma}\\
\le & \frac{1}{2}[F_{xx}]_{C^{0,\gamma}([0,T];{\rm Bil}(E))}[y]_{C^{\eta}([s,t];E)}\sum_{j=1}^{m_n}|s^n_j-s^n_{j-1}|^{(2+\gamma)\eta}\\
\le &
\frac{1}{2}[F_{xx}]_{C^{0,\gamma}([0,T];{\rm Bil}(E))}[y]_{C^{\eta}([s,t];E)}|\Pi_n(s,t)|^{(2+\gamma)\eta-1}\sum_{j=1}^{m_n}(s^n_j-s^n_{j-1})\\
=& \frac{1}{2}[F_{xx}]_{C^{0,\gamma}([0,T];{\rm Bil}(E))}[y]_{C^{\eta}([s,t];E)}|\Pi_n(s,t)|^{(2+\gamma)\eta-1}|t-s|,
\end{align*}
so that 
\begin{equation}
\lim_{n\to\infty}I_n^4=0.
\label{limit-I4}
\end{equation}

Finally, using formula \eqref{rappr_int_sol} we can write 
\begin{align*}
y(s^n_j)-y(s^n_{j-1})=\int_{s^n_{j-1}}^{s^n_j}Ay(u)du
+{\mathscr I}_{Gy}(s^n_{j-1},s^n_j)
\end{align*}
for every $j=1,\ldots,m_n$, so that
\begin{align*}
I_3^n+I_5^n=&\sum_{j=1}^{m_n}\bigg\langle F_x(s_{j-1}^n,y(s^n_{j-1})),
\int_{s^n_{j-1}}^{s^n_j}Ay(u)du
\bigg\rangle\\
&+\sum_{j=1}^{m_n}\langle F_x(s_{j-1}^n,y(s^n_{j-1})),
\mathscr{I}_{Gy}(s^n_{j-1},s^n_j)\rangle\\
&+\frac{1}{2}\sum_{j=1}^{m_n}\bigg\langle F_{xx}(s^n_{j-1},y(s^n_{j-1}))
\int_{s^n_{j-1}}^{s^n_j}Ay(u)du,\int_{s^n_{j-1}}^{s^n_j}Ay(u)du\bigg\rangle\\
&+\sum_{j=1}^{m_n}\bigg\langle F_{xx}(s^n_{j-1},y(s^n_{j-1}))\int_{s^n_{j-1}}^{s^n_j}Ay(u)du,\mathscr{I}_{Gy}(s^n_{j-1},s^n_j)\bigg\rangle\\
&+\frac{1}{2}\sum_{j=1}^{m_n}\langle F_{xx}(s^n_{j-1},y(s^n_{j-1}))\mathscr{I}_{Gy}(s^n_{j-1},s^n_j),\mathscr{I}_{Gy}(s^n_{j-1},s^n_j)\rangle=\sum_{k=1}^5J_k^n.
\end{align*}

Note that 
\begin{align}
J_1^n=&\sum_{j=1}^{m_n}\langle F_x(s_{j-1}^n,y(s^n_{j-1})),Ay(s^n_{j-1})\rangle(s^n_j-s^n_{j-1})\notag\\
&+\sum_{j=1}^{m_n}\bigg\langle F_x(s_{j-1}^n,y(s^n_{j-1})),
\int_{s^n_{j-1}}^{s^n_j}(Ay(u)-Ay(s^n_{j-1}))du\bigg\rangle
\label{pamela-1}
\end{align}
and the second term in the right-hand side of the previous formula vanishes as $n$ tends to $\infty$. Indeed, from Remark \ref{rmk:cont_y_D(A)}, the function $Ay$ is continuous in $[s,T]$ with values in $E$. Hence, it is uniformly continuous in $[s,t]$ and the term
$|Ay(u)-Ay(s^n_{j-1})|_E$ can be made arbitrarily small provided that $|\Pi_n(s,t)|$ is small enough. The arguments used 
to estimate the term $I_1^n$ allow us to conclude that the second term in the right-hand side of \eqref{pamela-1} vanishes as $n$ tends to $\infty$. Since the function 
$s\mapsto \langle F_x(s,y(s)),Ay(s)\rangle$ is continuous in $[s,t]$, we can infer that 
\begin{align*}
\lim_{n\to +\infty}J_1^n=\int_s^t\langle F_x(u,y(u)),Ay(u)\rangle du.
\end{align*}

Now, we consider the term $J_2^n$ and observe that
\begin{align*}
J^n_2=&
\sum_{j=1}^{m_n}\langle F_x(s_{j-1}^n,y(s^n_{j-1})),
Gy(s^n_{j-1})\rangle (x(s^n_j)-x(s^n_{j-1}))\\
&+\sum_{j=1}^{m_n}\langle F_x(s_{j-1}^n,y(s^n_{j-1})),
G^2y(s^n_{j-1})\rangle\mathbb{X}(s
^n_{j-1},s^n_j)\\
&-\sum_{j=1}^{m_n}\langle F_x(s_{j-1}^n,y(s^n_{j-1})),
M_{g_{G^2y}}(s^n_{j-1},s^n_j)\rangle
\end{align*}
and the last term in the right-hand side of the previous formula vanishes as $n$ tends to $\infty$.

As far as the term $J_3^n$ is concerned, using the continuity of the function $Ay$ in $[s,t]$ we can estimate
\begin{align*}
J_3^n\le & \frac{1}{2}
\sum_{j=1}^{m_n}\|F_{xx}\|_{C([0,T]\times E;{\rm Bil}(E))}\|Ay\|_{C([s,t];E)}^2\sum_{j=1}^{m_n}(s^n_j-s^n_{j-1})^2\\
\le & 
\frac{1}{2}\sum_{j=1}^{m_n}\|F_{xx}\|_{C([0,T]\times E;{\rm Bil}(E))}\|Ay\|_{C([s,t];E)}^2|\Pi_n(s,t)|(t-s),
\end{align*}
so that $J_3^n$ vanishes as $n$ tends to $\infty$.

Similarly, the term $J_4^n$ vanishes as $n$ tends to $\infty$. Indeed, 
\begin{align}
\int_{s^n_{j-1}}^{s^n_j}Gy(u)dx(u)=&
Gy(s^n_{j-1})(x(s^n_j)-x(s^n_{j-1}))\notag\\
&+\frac{1}{2}G^2y(s^n_{j-1})(x(s^n_j)-x(s^n_{j-1}))^2-
M_{g_{G^2y}}(s^n_{j-1},s^n_j),
\label{pamela-2}
\end{align}
so that its $E$-norm can be estimated from above in term of a positive constant times $|s^n_j-s^n_{j-1}|^{\eta}$, and we get
$|J_4^n|_E\le C|\Pi_n(s,t)|^{\eta}(t-s)$
for a positive constant $C$, independent of $n$.

Further, using \eqref{pamela-2} and the fact that $\mathbb X(\sigma,\tau)=\displaystyle\frac12(x(\tau)-x(\sigma))^2$ for every $(\sigma,\tau)\in[0,T]^2_<$, we can write
\begin{align*}
J_5^n=&\sum_{j=1}^{m_n}\langle F_{xx}(s^n_{j-1},y(s^n_{j-1}))Gy(s^n_{j-1}),Gy(s^n_{j-1})\rangle \mathbb{X}(s^n_{j-1},s^n_j)\\
&+\frac{1}{2}\sum_{j=1}^{m_n}\langle F_{xx}(s^n_{j-1},y(s^n_{j-1}))Gy(s^n_{j-1}),G^2y(s^n_{j-1})\rangle
(x(s^n_j)-x(s^n_{j-1}))^3\\
&-\sum_{j=1}^{m_n}\langle F_{xx}(s^n_{j-1},y(s^n_{j-1}))Gy(s^n_{j-1}),M_{g_{G^2y}}(s^n_{j-1},s^n_j)\rangle(x(s^n_j)-x(s^n_{j-1}))\\
&-\sum_{j=1}^{m_n}\langle F_{xx}(s^n_{j-1},y(s^n_{j-1}))G^2y(s^n_{j-1}),M_{g_{G^2y}}(s^n_{j-1},s^n_j)\rangle \mathbb{X}(s^n_{j-1},s^n_{j})\\ 
&+\frac{1}{2}\sum_{j=1}^{m_n}\langle F_{xx}(s^n_{j-1},y(s^n_{j-1}))G^2y(s^n_{j-1}),G^2y(s^n_{j-1})\rangle (\mathbb{X}(s^n_{j-1},s^n_{j}))^2\\
&+\frac{1}{2}\sum_{j=1}^{m_n}\langle F_{xx}(s^n_{j-1},y(s^n_{j-1}))M_{g_{G^2y}}(s^n_{j-1},s^n_j),M_{g_{G^2y}}(s^n_{j-1},s^n_j)\rangle.
\end{align*}

Since $x\in C^{\eta}([0,T])$, with $\eta\in \left (\frac{1}{3},\frac{1}{2}\right ]$ the second- and fifth-terms in the previous chain of equalities vanish as $n$ tends to $\infty$. Similarly, the third-, the fourth- and the last-terms vanish as $n$ tends to $\infty$ since $M_{g_{G^2y}}\in \mathscr{C}^{2\eta+\alpha}([0,T]^2_{<};E)$ and $2\eta+\alpha>1$. 

Combining all the above results, we conclude that, for every $n\in\N$,
\begin{align}
I_3^n+I_5^n
=&a_n+\int_s^t\langle F_x(u,y(u)),Ay(u)\rangle du\notag\\
&+\sum_{j=1}^{m_n}\langle F_x(s_{j-1}^n,y(s^n_{j-1})),
Gy(s^n_{j-1})\rangle (x(s^n_j)-x(s^n_{j-1}))\notag\\
&+\sum_{j=1}^{m_n}\langle F_x(s_{j-1}^n,y(s^n_{j-1})),
G^2y(s^n_{j-1})\rangle\mathbb{X}(s
^n_{j-1},s^n_j)\notag\\
&+
\sum_{j=1}^{m_n}\langle F_{xx}(s^n_{j-1},y(s^n_{j-1}))Gy(s^n_{j-1}),Gy(s^n_{j-1})\rangle \mathbb{X}(s^n_{j-1},s^n_j), 
\label{limit-mesh}
\end{align}
where $(a_n)_{n\in\N}$ is a suitable sequence which vanishes as $n$ tends to $\infty$.

Now, we prove that the sum of the last three terms in the right-hand side of \eqref{limit-mesh} converges to the rough integral of the function $h$, defined by 
$h(r)=\langle F_x(r,y(r)),Gy(r)\rangle$ for every $r\in [0,T]$, over the interval $[s,t]$. For this purpose, we begin by proving that $h$ admits Gubinelli derivative. A simple computation reveals that 
\begin{align*}
&(\delta h)(r_1,r_2)\\
=& \langle F_x(r_2,y(r_2))-F_x(r_1,y(r_2)),Gy(r_2)\rangle
+\langle F_x(r_1,y(r_2))-F_x(r_1,y(r_1)),Gy(r_2)\rangle\\
&+\langle F_x(r_1,y(r_1)),Gy(r_2)-Gy(r_1)\rangle\\
=&\langle F_x(r_2,y(r_2))-F_x(r_1,y(r_2)),Gy(r_2)\rangle
+\langle F_{xx}(r_1,\widetilde y_{12})(\delta y)(r_1,r_2),Gy(r_2)\rangle\\
&+\langle F_x(r_1,y(r_1)),G^2y(r_1)\rangle (x(r_2)-x(r_1))+\langle F_x(r_1,y(r_1)),GR^Y(r_1,r_2)\rangle\\
=&\langle F_x(r_2,y(r_2))-F_x(r_1,y(r_2)),Gy(r_2)\rangle\\
&+\langle F_{xx}(r_1,\widetilde y_{12})Gy(r_1),Gy(r_2)\rangle (x(r_2)-x(r_1))
+\langle F_{xx}(r_1,\widetilde y_{12})R^Y(r_1,r_2),Gy(r_2)\rangle\\
&+\langle F_x(r_1,y(r_1)),G^2y(r_1)\rangle (x(r_2)-x(r_1))+\langle F_x(r_1,y(r_1)),GR^Y(r_1,r_2)\rangle\\
=&\langle F_{xx}(r_1,\widetilde y_{12})Gy(r_1),Gy(r_2)\rangle (x(r_2)-x(r_1))\\
&+\langle F_x(r_1,y(r_1)),G^2y(r_1)\rangle (x(r_2)-x(r_1))
+o(|r_2-r_1|^{\zeta_1}),
\end{align*}
where $\zeta_1=\min\{\sigma,\eta+\alpha\}$ and $\widetilde r_{12}$, $\widetilde y_{12}$ are suitable points, in the segments joining $r_1$ to $r_2$ and $y(r_1)$ to $y(r_2)$, respectively.
Indeed, 
\begin{align*}
&|\langle F_x(r_2,y(r_2))-F_x(r_1,y(r_2)),Gy(r_2)\rangle|\\
\le &\|F_x\|_{C^{\sigma,0}([0,T]\times E;\mathscr{L}(E))}\|G\|_{\mathscr{L}(E)}\|y\|_{B([0,T];E)}|r_2-r_1|^{\sigma}\\[1mm]
&|\langle F_{xx}(r_1,\widetilde y_{12})R^Y(r_1,r_2),Gy(r_2)\rangle |\\
\le &\|F_{xx}\|_{C_b([0,T]\times E;{\rm Bil}(E))}\|R^Y\|_{\mathscr{C}^{\eta+\alpha}([0,T]^2_{<};E)}\|G\|_{\mathscr{L}(E)}\|y\|_{B([0,T];E)}|r_2-r_1|^{\eta+\alpha},\\[1mm]
&|\langle F_x(r_1,y(r_1)),GR^Y(r_1,r_2)\rangle|\\
\le &
\|F_x\|_{C_b([0,T]\times E;\mathscr{L}(E))}\|G\|_{\mathscr{L}(E)}\|R^Y\|_{\mathscr{C}^{\eta+\alpha}([0,T]^2_{<};E)}|r_2-r_1|^{\eta+\alpha}.
\end{align*}

Next, we observe that
\begin{align*}
&\langle F_{xx}(r_1,\widetilde y_{12})Gy(r_1),Gy(r_2)\rangle\\
=&\langle F_{xx}(r_1,y(r_1))Gy(r_1),Gy(r_2)\rangle+\langle [F_{xx}(r_1,\widetilde y_{12})-F_{xx}(r_1,y(r_1))]Gy(r_1),Gy(r_2)\rangle\\
=&\langle F_{xx}(r_1,y(r_1))Gy(r_1),Gy(r_1)\rangle+\langle F_{xx}(r_1,y(r_1))Gy(r_1),Gy(r_2)-Gy(r_1)\rangle\\
&+\langle (F_{xx}(r_1,\widetilde y_{12})-F_{xx}(r_1,y(r_1))Gy(r_1),Gy(r_2)\rangle\\
=&\langle F_{xx}(r_1,y(r_1))Gy(r_1),Gy(r_1)\rangle
+O(|r_2-r_1|^{\gamma})+O(|r_2-r_1|^{\alpha}),
\end{align*}
since $F_{xx}\in C^{0,\gamma}([0,T];{\rm Bil}(E))$, 
$\hat\delta_1Gy\in \mathscr{C}^{\alpha}([0,T]^2_{<};E_{\alpha})$ and $Gy\in  B([0,T];E_{\alpha})$, which implies that $Gy\in C^{\alpha}([0,T];E)$. 
Putting everything together, we have proved that
\begin{align*}
(\delta h)(r_1,r_2)=&[\langle F_x(r_1,y(r_1)),G^2y(r_1)\rangle\\
&+\langle F_{xx}(r_1,y(r_1))Gy(r_1),Gy(r_1)\rangle](x(r_2)-x(r_1))+O(|t-s|^{\rho}),
\end{align*}
where $\rho=\min\{\sigma,\eta+\alpha,\gamma+\eta\}$. Note that $\rho+\eta>1$ since $\sigma>1-\eta$, by assumptions, $2\eta+\alpha>1$ and $\gamma+2\eta>1$ follows from the condition $2\eta+\gamma\eta>1$.

Hence, the function $r\mapsto \langle F_x(r,y(r)),G^2y(r)\rangle+\langle F_{xx}(r,y(r))Gy(r),Gy(r)\rangle$ is a Gubinelli derivative of
the function $h$.

Since $\check h\in C^{\xi}([0,T];E)$ with $\xi=\min\{\sigma,\lambda,\eta\gamma\}$ and $\lambda\in (1-\eta,1)$, it follows immediately that $\xi+2\eta>1$. Hence,
the rough integral
$\int_s^t\langle F_x(u,y(u)),Gy(u)\rangle dx(u)$ is well-defined and 
\begin{align*}
&\int_s^t\langle F_x(u,y(u)),Gy(u)\rangle dx(u)\\
=&
\lim_{n\to\infty}\bigg [\sum_{j=1}^{m_n}\langle F_x(s_{j-1}^n,y(s^n_{j-1})),
Gy(s^n_{j-1})\rangle (x(s^n_j)-x(s^n_{j-1}))\\
&\qquad\qquad\quad+\sum_{j=1}^{m_n}\langle F_x(s_{j-1}^n,y(s^n_{j-1})),
G^2y(s^n_{j-1})\rangle\mathbb{X}(s
^n_{j-1},s^n_j)\\
&\qquad\qquad\quad+
\sum_{j=1}^{m_n}\langle F_{xx}(s^n_{j-1},y(s^n_{j-1}))Gy(s^n_{j-1}),Gy(s^n_{j-1})\rangle \mathbb{X}(s^n_{j-1},s^n_j)\bigg ].
\end{align*}

Letting $n$ tend to $\infty$ in \eqref{limit-mesh}, we can infer that
\begin{align*}
\lim_{n\to +\infty}(I_3(n)+I_5(n))=&
\int_s^t\langle F_x(u,y(u)),Ay(u)\rangle du\\
&+\int_s^t\langle F_x(u,y(u)),Gy(u)\rangle dx(u).
\end{align*}

From this formula, \eqref{lim-I1}, \eqref{lim-I2} and \eqref{limit-I4}, the assertion follows for $s>0$. Letting $s$ tend to $0$ in \eqref{ito_formula}, and taking into account Lemma \ref{lem:esistenza_int_rough}(ii), estimate \eqref{estim-1} and recalling that $F$ and $y$ are continuous functions, we can easily complete the proof. 
\end{proof}

\section{An example}
\label{sect-6}
Let $\mathcal A$ be the second-order differential operator on $\R^d$ defined by
\begin{align*}
\mathcal A=\sum_{i,j=1}^dq_{ij}D^2_{ij}+\sum_{i=1}^dD_{i}+c, 
\end{align*}
with coefficients which belongs to $C^{\beta}_b(\R^d)$ for some $\beta\in (0,1)$. Further, assume that $\sum_{i,j=1}^dq_{ij}(x)\xi_i\xi_j\geq q_0|\xi|^2$ for every $x,\xi\in \R^d$ and some positive constant $q_0$. Let $A$ be the realization of $\mathcal A$ in $E=C_b(\R^d)$ with domain
\begin{align*}
D(A)=\bigg\{u\in C_b(\R^d)\cap\bigcap_{p<\infty}W^{2,p}_{\rm loc}(\R^d):\mathcal A u\in C_b(\R^d)\bigg\}.    
\end{align*}
It is well-known that $A$ generates an analytic semigroup $(S(t))_{t\geq0}$. For $\lambda\in [0,3)\setminus\{1,2\}$, we can take $E_{\lambda}=D_A(\lambda,\infty)$, $E_1=D(A)$ and $E_2=D(A^2)$. In particular, $E_{\lambda}=C^{2\lambda}_b(\R^d)$ if $\lambda\in (0,1)\setminus\{\frac{1}{2}\}$ and $E_{1/2}$ is the usual Zygmund space 
(see \cite[Chapter 14]{lor-rha}).
Let $x$ be a trajectory of a fractional Brownian motion with Hurst parameter $H\in\left(\frac13,\frac12\right]$ and let $g\in C^{2\rho}_b(\R^d)$ for some $\rho\in(0,1)$ such that $H+\rho>1$. Then, the linear operator $G:E\to E$, defined as $(Gy)(\xi)=g(\xi)y(\xi)$ for every $y\in C_b(\R^d)$ and $\xi\in\R^d$, maps $E_\lambda$ into $E_\lambda$ for every $\lambda\in[0,\rho]$. In particular, for every choice of $\mathbb{X}\in\mathscr{C}^{2H}([0,T]^2_{<})$, which satisfies condition \eqref{form-diff-X}, and every $\psi\in E_{H+\alpha}$, with $\alpha\in\left(1-2H,\min\{H,\rho-H\}\right )$, problem \eqref{cauchy_prob_rough} admits a unique mild solution $y$ such that $y(t)\in D(A)$ for every $t\in(0,T]$.

\end{document}